\documentclass{article}
\usepackage[a4paper,left=2.5cm,right=2.0cm,top=2.5cm,bottom=2.0cm]{geometry}
\usepackage[dvips]{graphicx}
\usepackage[latin1]{inputenc}
\usepackage{amsfonts}
\usepackage{amsmath}
\usepackage[mathcal]{euscript}
\usepackage{color}
\usepackage{url}
\usepackage[all]{xy}
\newtheorem{proposicao}{Proposition}
\newtheorem{teorema}{Theorem}
\newtheorem{lema}{Lemma}
\newtheorem{corolario}{Corollary}
\newenvironment{proof}{{\noindent\bf Proof.}}{\hfill\vspace{0.1cm}}

\newtheorem{observacao}{Remark}
\newtheorem{definicao}{Definition}
\begin{document}
\title{A nonsmooth optimization technique in domains of positivity\footnote{{\scriptsize This work was supported by FAPERJ, Brazil.}}\author{Ronaldo Malheiros Greg\'orio\footnote{{\scriptsize Corresponding Author. Universidade Federal Rural do Rio de Janeiro. Programa de P\'os-gradua\c c\~ao em Modelagem Matem\'atica e Computacional. Av. Gov. Roberto Silveira, s/n, Moquet\'a, Nova Igua\c cu - CEP 26020-740, RJ, Brazil. E-mail: rgregor@ufrrj.br.}} \and }
\date{\today}}
\maketitle
\vspace{-1.0cm}

\begin{abstract}

This paper presents a nonsmooth proximal point technique for convex optimization in a special class of Hadamard manifold called homogeneous domains of positivity. The method is based on the particularization of the Rham decomposition theorem to symmetric spaces and it is applicable to the computing of minimizers for convex functions. Homogeneous domains of positivity  to be considered in this work  are those of specially reducible type in a sense which is introduced and largely discussed along the paper. General aspects on the technique are shown as the convergence analysis to the inner iterations of the method, the global convergence for both versions exact and inexact and, under particular assumptions, a lower bound to the number of outer iterations of the inexact version.        
\vspace{0.2cm}


\vspace{0.2cm}



\end{abstract}
\vspace{-0.5cm}
\section{Introduction}

In several situations researchers need to estimate solutions for optimization problems on manifolds of non-Euclidean type and this practice brings together  two other relevant aspects, the development of new optimization techniques and the extension of classical  methods. In this paper an optimization technique which is at the same time an extension and a theoretical improvement of the proximal point algorithm presented in \cite{Gregorio} to domains of positivity is introduced. The method is essentially based on the idea of the Rham Decomposition Theorem for simply connected symmetric spaces. According to \cite[p. 180]{Sakai}, the Rham Theorem can be enunciated as follows
\begin{teorema}[Rham decomposition theorem for simply connected symmetric spaces]
Let $S$ be a simply connected symmetric space. Then $S$ can be rewritten as the Cartesian product of a finite number of simply connected symmetric spaces $S_1,\cdots, S_N$ of irreducible type and a Euclidean space $S_0$, i.e., $$S=\left(\displaystyle\prod_{\iota=1}^{N}S_{\iota}\right)\times S_0=\left(S_1\times\cdots\times S_N\right)\times S_0.$$ Moreover, the isometry group  $I_0(S)$ of $S$ is also decomposed as $I_0(S)=\left(\displaystyle\prod_{\iota=1}^{N}I_0(S_{\iota})\right)\times I_0(S_0)$, where $I_0(S_{\iota})$ is the isometry group of $S_{\iota} (\ \iota=0,1,\cdots,N$). \label{Rham}  
\end{teorema}

The current algorithm is applicable to estimation of  minimizers for convex functions in homogeneous domains of positivity of \textit{specially reducible} type. Specially reducible domains are introduced and largely discussed in the next section. Specifically, this work discourses about aspects such as the establishment of the method, the convergence of their inner iterations and the global convergence for both  exact and inexact versions.  

In the following, an extensive narrative that reports to theoretical aspects, algorithms and applications from the related literature which have inspired the development of the current method is presented. The next section is dedicated to the establishment of both the problem and the technique, by ensuring  well-posedness and convergence of their inner iterations. Furthermore, the global convergence for the exact version to the method is still discussed at the end of the Section 2. Section 3 discourses about the inexact version to the technique. There, it is showed that globally the algorithm in \cite{Gregorio} and the current technique derive from the same extension proposed in \cite{Ferreira2}. A particular analysis on the lower bound to the number of outer iterations for the method under certain assumptions closes the Section 3. Section 4 finishes the main content of the paper with remarks about practical aspects related to the implementation of the method. 
 
\subsection{Related works}

Synthesizing extensions of optimization methods to Riemannian manifolds can be as complicated as the establishment of new tools. In general, extensions are not easy to be obtained formally. Their convergences do not derive from a straightforward argumentation. Nothing different than naturally happens with classical optimization methods in Euclidean spaces where preliminary auxiliary results are separately ensured before to establish the main convergence theorem. Besides, a deep knowledge on Topology of Manifolds and an arduous study on concepts from  Differential Geometry are demanded by anyone who is interested in optimization on manifolds. For instance, the methods in \cite{Ferreira1}, \cite{Smith} and \cite{Udriste} exemplify nontrivial extensions of classical optimization techniques to complete Riemannian manifolds.

Although recent, some extensions have already been reason of inspiration to sophisticated optimization methods. Indeed, the replacement of the Riemannian distance in the definition of the extended proximal point iteration in \cite{Ferreira2} by Bregman distances enabled authors in \cite{Quiroz} to develop a new class of Riemannian proximal point methods for generalized convex optimization problems on noncompact Hadamard manifolds. In other cases, they enable the development of specific tools for particular  manifolds. For instance, authors in \cite{Gregorio}  used the extension in \cite{Ferreira2} to develop a proximal point technique able to estimate minimizers for convex functions whose domain is the cone of $n$-by-$n$ symmetric positive semidefinite matrices and  whose iterations belong to the interior of this cone. Namely, the interior of the cone of $n$-by-$n$ symmetric positive semidefinite matrices is denoted by $\mathbb{P}_{n}$ and it is called \textit{manifold of $n$-by-$n$ symmetric positive definite matrices}. See \cite{Bhatia} for a long discussion about the Riemannian structure to $\mathbb{P}_{n}$.    

Developments with symmetric positive definite matrices that regards the Riemannian geometry to $\mathbb{P}_{n}$ are often found in the scientific literature. As known, $\mathbb{P}_{n}$ is a particular homogeneous domain of positivity and, when it is provided with an appropriate Riemannian metric, it is a manifold of Hadamard type. So, it justifies the interest of researchers from areas like Optimization on Manifolds in that cone. 

Metric, geodesics and other Riemannian features in $\mathbb{P}_{n}$ have already been highlighted by a large number of authors. For instance, see \cite{Bhatia}, \cite{Moakher} and \cite{Nesterov}. Analytical expressions to geodesics in $\mathbb{P}_{n}$ are uniquely synthesized whether initial conditions are prefixed for example. Indeed, denote by $\mathbb{S}_{n}$ the vector space of symmetric matrices. Given a starter point $x\in\mathbb{P}_{n}$ and a velocity $s\in\mathbb{S}_{n}$, the expression to the geodesic $\gamma$ in $\mathbb{P}_{n}$ that passes by $x$ with velocity $s$ is given by  
$$\gamma(t)=x^{\frac{1}{2}}e^{tx^{-\frac{1}{2}}sx^{-\frac{1}{2}}}x^{\frac{1}{2}}, \ \ t\in \mathbb{R}.$$
The parametrization to the curve $\gamma_{xy}$ of the smallest size that connects a pair of points $x,y\in\mathbb{P}_{n}$ has a closed rule as happens in the Euclidean case. It is given by   
$$\gamma_{xy}(t)=x^{\frac{1}{2}}(x^{-\frac{1}{2}}yx^{-\frac{1}{2}})^tx^{\frac{1}{2}}, \ \ t\in [0,1].$$ 
Moreover, the distance between $x$ and $y$ resulting from the Riemannian measure to arc length of curves $$d(x,y)=\displaystyle\inf_{c\in{\cal C}_{xy}}\int^{1}_{0}\|\dot{c}(t)\|_{c(t)}dt$$ is achieved by the length of the geodesic segment $\gamma_{xy}$, where ${\cal C}_{xy}$ is the set of all regular curves $c$ in $\mathbb{P}_{n}$  that connect $x$ to $y$, i.e., $c(0)=x$, $c(1)=y$ and $c(t)\in\mathbb{P}_{n}$, with $\dot{c}(t)\not={\bf 0}$\footnote{${\bf 0}$ denotes the null vector from the space of $n$-by-$n$ symmetric matrices}, for every $t\in(0,1)$. Particularly, the length of $\gamma_{xy}$ is given by
$$d(x,y)=\sqrt{\sum_{\iota=1}^{n}{\rm ln}^2 \ \lambda_{\iota}(x^{-\frac{1}{2}}yx^{-\frac{1}{2}})},$$ 
where $\lambda_{\iota}(x^{-\frac{1}{2}}yx^{-\frac{1}{2}})$ is the $\iota$th eigenvalue of $x^{-\frac{1}{2}}yx^{-\frac{1}{2}}$ ($\iota=1,\cdots, n$).

Exponential and logarithmic functions of matrices above do not seem natural to starters or researchers from fields which do not intersect target areas discussed in this paper. However, they are naturally disseminated by authors from areas like Differential Geometry, specially those who investigate Matrix Manifolds. Working with real functions of symmetric matrices for example is so easier than one can imagine. The Schur decomposition for symmetric matrices is claimed to this aim. Indeed, let $h:\mathbb{R}\to\mathbb{R}$ be any real function. Since  for every $x\in\mathbb{S}_{n}$ there exist a $n$-by-$n$ orthogonal matrix $w$ and a $n$-by-$n$ real diagonal matrix $\lambda$ such that $x=w\lambda w^T$, $h(x)=wh(\lambda)w^T$, where $h(\lambda)$ is the real diagonal matrix whose diagonal elements are of the form $[h(\lambda)]_{\iota\iota}=h(\lambda_{\iota\iota})$  ($\iota=1,\cdots, n$). See \cite{Golub} for details about definition and properties to functions of matrices.   

A key concept in global optimization is convexity. Let $\mathbb{E}$ be any finite dimensional Euclidean space. Here, a subset $C$ of a complete Riemannian manifold $\mathbb{M}\subset \mathbb{E}$ is called convex if it contains all minimal geodesic segments connecting any pair of their points. Besides, a function $f:C\to\mathbb{R}$ is said to be convex if 
$$(f\circ\gamma_{xy})((1-t)\cdot t_1+t\cdot t_2)\leq (1-t)\cdot f(x)+t\cdot f(y),$$
for every $x,y\in C$ and $t\in [0,1]$, where $\gamma_{xy}$ is any minimal geodesic segment in $C$ satisfying $\gamma_{xy}(t_1)=x$ and $\gamma_{xy}(t_2)=y$. Moreover, $f$ is said to be strictly convex if the inequality above is strict, for every $t\in (0,1)$.  See \cite{Sakai} and \cite{Udriste} for definition and examples of convex functions in Riemannian manifolds. In particular, for manifolds of Cartan-Hadamard type, see \cite{Shiga}. That last paper discourses about Hadamard manifold only. So, it is intrinsically related to the current work since here only  homogeneous domains of positivity are considered and, as shown in \cite{Rothaus}, they are Riemannian manifolds of sectional curvature everywhere nonpositive.

Almost all algebraical and topological properties for convex sets and functions as well as differential properties for convex functions in euclidean spaces have already been extended to convex sets and functions in complete Riemannian manifolds. For instance, subdifferentials of a convex functions remain nonempty at any point from its domain and a minimum is characterized as that point where the subdifferential contains the null vector. Namely, the subdifferential of a function $f$, at $x\in C$, denoted by $\partial f(x)$, is the set defined by 
\begin{equation}
\partial f(x):=\{s\in\mathbb{E}:f(y)\geq f(x)+\langle s,{\rm exp}^{-1}_xy \rangle_{x},\forall y\in C\}, \label{subdiferencial}
\end{equation}
where ${\rm exp}^{-1}_xy$ is the velocity of any minimal geodesic segment that connects $x$ to $y$ in $C$, at $x$, and $\langle,\rangle_{x}$ is the Riemannian metric on $\mathbb{M}$. Any element of $\partial f(x)$ is called a \textit{subgradient} to $f$, at $x$. If $f$ is differentiable then $\partial f(x)$ is an unitary set that contains only the Riemannian gradient\footnote{The Riemannian gradient of $f$, at $x$, denoted by ${\rm grad} \ f(x)$, is defined as the vector field that is metrically equivalent to $\nabla f(x)$, i.e., $\langle {\rm grad} \ f(x),s\rangle_{x}=\langle \nabla f(x),s\rangle$, for every $s\in T_x\mathbb{M}$, where $T_x\mathbb{M}$ is the tangent plane to $\mathbb{M}$, at $x$. If $G$ is the symmetric positive definite bi-linear form for which $\langle s_1,s_2\rangle_{x}=\langle G(x)s_1,s_2\rangle$, for every $x\in\mathbb{M}$ and $s_1,s_2\in T_x\mathbb{M}$, then ${\rm grad} \ f(x)=[G(x)]^{-1}\nabla f(x)$.} of $f$, at $x$, for any $x\in C$. See \cite{Udriste} for further details about convex analysis on Riemannian manifolds.

The state of the art in convex optimization methods is the proximal point algorithm introduced in \cite{Martinet} and extended to operators in \cite{Rockafellar}. In particular, the proximal point algorithm in \cite{Ferreira2} is also an extension of that method but now to Hadamard manifolds. It can be described as follows: admit that $\mathbb{M}$ is now a Hadamard manifold. Denote by ${\rm arg}\displaystyle\min_{x\in \mathbb{M}} \ f(x)$ the set of local minimizers  to $f$ in $\mathbb{M}$, for any arbitrary function $f:\mathbb{M}\to\mathbb{R}$. Also, admit that $f:\mathbb{M}\to \mathbb{R}$ is convex. Given  $x_{0}\in \mathbb{M}$ and a bounded sequence $\{\beta_{0}\}$ of positive real numbers, the proximal point algorithm generates a sequence $\{x_{k}\}$ defined by the iteration
\begin{equation}
x_{k+1}={\rm arg}\min_{y\in\mathbb{M}}\left\{f(y)+\frac{\beta_{k}}{2}d^2(y,x_{k})\right\} \ \ (k=0,1,\cdots), \label{iteracao_principal}
\end{equation}
for which $\{f(x_{k})\}$ converges to $\displaystyle\inf_{x\in\mathbb{M}} f(x)$. Moreover, $\{x_{k}\}$ converges to any point from ${\rm arg}\displaystyle\min_{x\in \mathbb{M}} \ f(x)$ whether it is nonempty.
 
An interesting variant from that method was proposed in \cite{Gregorio} to compute minimizers of convex functions in the closure of $\mathbb{P}_{n}$. It is based on the Schur decomposition theorem for symmetric positive definite matrices. Denote by $\mathbb{O}_n$ and $\mathbb{D}_{n}$ the sets of $n$-by-$n$ orthogonal and diagonal positive definite matrices, respectively. Let $f:\overline{\mathbb{P}_{n}}\to\mathbb{R}$ be convex in $\mathbb{P}_{n}$. Define $\phi_{w}:\mathbb{D}_{n}\to\mathbb{R}$ by $$\phi_{w}(\lambda)=f(x^{k^\frac{1}{2}}w\lambda w^Tx^{k^\frac{1}{2}})$$ and $\rho_{w}:\mathbb{D}_{n}\to\mathbb{R}$ by $$\rho_{w}(\lambda)=\frac{1}{2} d^2(x_{k},x^{k^\frac{1}{2}}w\lambda w^Tx^{k^\frac{1}{2}}),$$ where $x_{k}$ is the $k$th iterated from the proximal point algorithm in \cite{Ferreira2} and $w\in\mathbb{O}_{n}$ is prefixed. The method in \cite{Gregorio} computes $x_{k+1}$, for any $k$ ($k=0,1,\cdots$), through the following steps:
\begin{itemize}
\item[($S_1$)] \textbf{input} $w_{0}\in \mathbb{O}_n$, $\lambda_{0}\in \mathbb{D}_{n}$ and \textbf{set} $j=0$;
\item[($S_2$)] \textbf{compute} $\overline{\lambda}_{j+1}={\rm arg}\displaystyle\min_{\lambda\in \mathbb{D}_{n}} \left\{\phi_{j}(\lambda)+\beta_{k} \rho_{j}(\lambda)\right\},$ \textbf{where} $\phi_{j}\equiv\phi_{w_{j}}$ and $\rho_{j}\equiv\rho_{w_{j}}$;  
\item[($S_3$)] \textbf{compute} $w_{j+1}\in \mathbb{O}_n$ and $\lambda_{j+1}\in\mathbb{D}_{n}$ such that $w_{j+1}^Tw_{j}\overline{\lambda}_{j+1}w_{j}^Tw_{j+1}=\lambda_{j+1};$ 
\item[($S_4$)] \textbf{update} $j=j+1$ and return to ($S_2$).  
\end{itemize}

Proposition 2 in \cite[p. 475]{Gregorio} shows that the sequence $\{y_{j}\}$, defined by the iteration $$y_{j}=x^{k^\frac{1}{2}}w_{j}\lambda_{j} w_{j}^Tx^{k^\frac{1}{2}},$$ converges to $x_{k+1}$ if $\overline{\lambda}_{j+1}\not=\lambda_{j}$, for every $j$ ($j=0,1,\cdots$). Still, if $\overline{\lambda}_{j+1}=\lambda_{j}$ then $w_{j+1}$ can be updated as the own iterated $w_{j}$ and the method stops prematurely. Also in \cite{Gregorio} is proposed permutations on the diagonal elements of $\lambda_{j}$ to overcome that obstacle. It can be achieved by permuting the columns of $w_{j}$ for example. As known, matrices of permutation are orthogonal and product of orthogonal matrices become orthogonal. Thus, by setting $w_{j+1}=w_{j}\cdot p$, for any matrix of permutation $p\in\mathbb{O}_{n}$, and $\lambda_{j+1}=w_{j+1}^Tw_{j}\lambda_{j}w_{j}^Tw_{j+1}$, it is hoped that the algorithm restarts. However, in the worst case this can not be performed in less than $n!$ attempts since a total of $n!$ permutations can be made with columns of $w_{j}$. Other disadvantages are pointed in the topics bellow. 
\begin{enumerate}
\item[($T_1$)] The step ($S_3$) does not produce decreases in the regularized objective function at (\ref{iteracao_principal}). Indeed, by ($S_3$), $w_{j}\bar{\lambda}_{j+1}w^T_{j}=w_{j+1}\lambda_{j+1}w^T_{j+1}$ . It follows that $$y_{j+1}=x^{k^{\frac{1}{2}}}w_{j}\bar{\lambda}_{j+1}w^T_{j}x^{k^{\frac{1}{2}}}=x^{k^{\frac{1}{2}}}w_{j+1}\lambda_{j+1}w^T_{j+1}x^{k^{\frac{1}{2}}}=\overline{y}_{j+1},$$ since $x_{k}$ is nonsingular. So, the value of the regularized objective function in (\ref{iteracao_principal}) at  $y_{j+1}$ is exactly the same as at $\overline{y}_{j+1}$. 
\item[($T_2$)] The updating from $w_{j}$ to $w_{j+1}$ in ($S_3$) does not use any information about $f$ as well as any feature  about the Riemannian structure of $\mathbb{O}_n$. For instance, the natural geometry for $\mathbb{O}_n$ is strongly explored in \cite{Fiori2}.
\item[($T_3$)] Each Schur decomposition in ($S_3$) has high computational cost whether $n$ is large.   
\end{enumerate}

Now, define $\varphi_{x}:\mathbb{D}_{n}\times \mathbb{O}_{n} \to \mathbb{P}_{n}$ by $$\varphi_{x}(\lambda, w)=x^{\frac{1}{2}}w\lambda w^Tx^{\frac{1}{2}},$$ for any $x\in\mathbb{P}_{n}$. It is easily seen that $\varphi_{x}$ is onto since $x$ is nonsingular. Note that $\varphi_{x}\equiv T_{x}\circ\varphi$, where $T_{x}$ is the nonsingular linear operator on the space of $n$-by-$n$ matrices, defined by $$T_{x}(y)=x^{\frac{1}{2}}yx^{\frac{1}{2}},$$ that maps $\mathbb{P}_{n}$ onto itself and $\varphi:\mathbb{D}_{n}\times \mathbb{O}_{n}\to \mathbb{P}_{n}$, given by $\varphi(\lambda,w)=w\lambda w^T$, is onto. Furthermore, $\mathbb{O}_{n}$ is a nonconnected compact Lie group. Its Riemannian structure have already been employed in applications related to Independent Component Analysis. Besides \cite{Fiori2}, \cite{Nishimori} also discourses about those applications. On the other hand, extensions of classical optimization methods to that manifolds have already been proposed too.  The commonest is an adaptation of the geodesic gradient method introduced in \cite{Luenberger}. 

\subsection{Applications} 

The growing number of tools related to optimization in Hadamard manifolds is justified by the increase of mathematical models for real problems based on those structures. Particularly, advances in mathematical theories have evidenced $\mathbb{P}_{n}$ as a manifold of Hadamard type for a special choice of Riemannian metric. See \cite{Nesterov} for further explanations. In addition, Signal Processing and Computer Vision are examples of fields that aggregate models based on symmetry and positive definiteness.        

Diffusion tensor imaging (DTI) is by far the most relevant field of application to symmetry and positive definiteness. Any numerical representation of 3D image in DTI is done by a large matrix with the property that all 3-by-3 submatrices, named voxels, are symmetric and positive definite. Specifically in DTI, image cleaning and smoothing procedures are frequently demanded since noises often occur in the caption of images by magnetic resonance machines. Some outer interference signs introduce noises on the image, represented by deficient voxels (asymmetric or indefinite 3-by-3 matrices) which make smudge in its visualization. So, it is used to minimize the variance of the imaging data from the noisy region.

The procedure above is known as \textit{weighted mean filtering}. Other kinds of filtering as those made by \textit{median} have already been proposed. Nowadays, filtering is classified in two groups, one that taking into account the Euclidean structure to $\mathbb{P}_{3}$ and the other, the Riemannian one. Recent papers show that Riemannian filtering has proven to be more efficient than the Euclidean one once Riemannian filters are generally smooth, with respect to the natural structure to $\mathbb{P}_{3}$. It is largely discussed by several authors. See  \cite{Casta}, \cite{Fiori1}, \cite{Fletcher} and  \cite{Zhang}. Weighted mean and median are the most relevant tools studied for this purpose and they are defined as solutions of smooth optimization problems in $\mathbb{P}_{3}$. The weighted mean is used by Gaussian smoothing in Riemannian filtering. This is described with details in \cite{Pennec}. Namely, the weighted mean is defined as the solution to the following smooth optimization model
$$
\begin{array}{cc}
{\rm minimize} & \frac{1}{2}\sum_{i=1}^{m}\omega_{i}d^2(x,x_{i})\\
{\rm subject \ to} & x\in \mathbb{P}_{3},
\end{array}
$$
and the median, as solution of another similar model, also smooth, given by  
$$
\begin{array}{cc}
{\rm minimize} & \sum_{i=1}^{m}d(x,x_{i})\\
{\rm subject \ to} & x\in \mathbb{P}_{3},
\end{array}
$$
where $m$ represents the amount of neighbors from the noisy voxel used to compute the weighted mean or the median. $\omega_{i}$ is a weight associated to the voxel $x_{i}$  ($i=1,\cdots, m,$) and $d:\mathbb{P}_{3}\times \mathbb{P}_{3}\to \mathbb{R}_{+}$ is the Riemannian distance resulting from the  metric defined by the Hessian of the standard logarithmic barrier to Semidefinite Programming. Several other applications are cited in \cite{Fiori1} nevertheless it is relevant to highlight  \textit{human detection via classification} which is discussed in details in \cite{Tuzel}. Covariances of symmetric positive definite matrices in $\mathbb{P}_5$ is also demanded by that kind of application.     

Another matrix manifold that keeps similarities with $\mathbb{P}_{n}$ is the set of $n$-by-$n$  Hermitian positive definite matrices. Here, it is denoted by $\mathbb{H}_{n}$. As a matter of fact, $\mathbb{H}_{n}$ extends all properties of $\mathbb{P}_{n}$ to the complex case. Hermitian positive definite matrices have application in studies related to modeling of covariance matrices and trigonometric moments of nonnengative functions for example. These applications are briefly described in \cite{Horn}. Covariance on $\mathbb{H}_{n}$, for example, also involves the computing of averages. An average is defined in a Riemannian sense as a solution of a smooth optimization problem on $\mathbb{H}_{n}$ in terms of the \textit{Karcher mean problem}. In the practice, the Karcher mean problem in $\mathbb{H}_{n}$ is built by choosing an appropriate Riemannian metric. The general Riemannian metric introduced in \cite{Rothaus} to homogeneous domains of positivity points how it can be made. 

\section{Problem statement and algorithm} 
This section starts with a list of preliminary concepts involving domains of positivity. All of them were adapted from \cite{Rothaus}. In addition, the meaning to the term \textit{specially reducible} is presented too. After that, the problem and the current technique are stated by describing different aspects as convergence of the inner iterations of the method as well as its global convergence.

Accurately, a domain of positivity is an open self-dual convex cone in any finite dimensional Euclidean space. See \cite{Koecher} for a synthesis on the theme. Moreover, according to \cite{Barut}, homogeneous domains of positivity are symmetric spaces with additional properties presented in the following. 
    
\label{secao_2}

\subsection{Preliminary}
Let $D$ be a nonempty open set in $\mathbb{E}$ and $\sigma:\mathbb{E}\times\mathbb{E}\to \mathbb{R}$ a nonsingular symmetric bi-linear form. 
\begin{definicao}
$D$ is called domain of positivity, with respect to $\sigma$, if the following axioms are guaranteed:
\begin{itemize}
\item[($A_1$)] $\sigma(a_1,a_2)>0$, for every $a_1,a_2\in D$;
\item[($A_2$)] if $a\in \mathbb{E}$ is such that $\sigma(a,\bar{a})>0$, for every $\bar{a}\in\overline{D}$ $(\bar{a}\not=0)$, where $\overline{D}$ is the closure of $D$, then $a\in D$.
\end{itemize}
Namely, $\sigma$ is called \textit{characteristic} of $D$.
\end{definicao}

Let ${\rm Tr}:\mathbb{R}^{n\times n}\to \mathbb{R}$ denote the trace function of square matrices. By putting $\sigma(x,y)={\rm Tr} \ \{xy^{*}\}$ (${\rm Tr} \ \{xy^{T}\}$), for any $x,y\in\mathbb{H}_n$ ($\mathbb{P}_{n}$), it follows that $\mathbb{H}_n (\mathbb{P}_n)$  is a domain of positivity with respect to $\sigma$. Indeed, it is ensured by the next results. The proofs of the first three lemmas are omitted once they are integrally found in \cite{Horn}, more precisely in the proof of Theorem 7.2.7, Corollary 7.2.11 and Theorem 7.5.3 in \cite[p. 406, 408, 458, resp.]{Horn}.
\begin{lema}
A matrix $x\in\mathbb{C}^{n\times n}$ ($\mathbb{R}^{n\times n}$) is positive definite iff there is a nonsingular matrix $y\in\mathbb{C}^{n\times n}$ ($\mathbb{R}^{n\times n}$) such that $x=y^{*}y$ ($x=y^Ty$).\label{Cholesky}
\end{lema}
Lemma \ref{Cholesky} asserts that positive definite matrices are necessarily Hermitian (symmetric).  
\begin{lema}
Let $x\in\mathbb{C}^{n\times n}$ ($\mathbb{R}^{n\times n}$) be an arbitrary square matrix. $x$ is positive semidefinite with rank $R\leq n$ iff there is a set of vectors $\{u_1,\cdots,u_n\}\subset \mathbb{C}^{n}$ ($\mathbb{R}^{n}$), containing exactly $R$ linearly independent vectors such that $$x_{\iota\tilde{\iota}}=\langle u_{\iota},u_{\tilde{\iota}}\rangle \ \ (\iota,\tilde{\iota}=1,\cdots, n).$$\label{Gram}
\end{lema}
\begin{lema}
Let $x,y\in\mathbb{C}^{n\times n}$ ($\mathbb{R}^{n\times n}$) be arbitrary square matrices. Denote by $x\odot y$ the Hadamard product of $x$ and $y$, i.e., $(x\odot y)_{\iota\tilde{\iota}}=x_{\iota\tilde{\iota}}y_{\iota\tilde{\iota}}$, ($\iota,\tilde{\iota}=1,\cdots, n$). If $x$ and $y$ are positive definite then so is $x\odot y$.  \label{Hadamard}
\end{lema}
\begin{teorema}
$\mathbb{H}_n$ is a domain of positivity with respect to $\sigma(x,y)={\rm Tr} \ \{xy^{*}\}$.
\end{teorema}
\begin{proof}
First, we need to prove that $\sigma$ attends the Axiom ($A_1$). In fact, by Lemma \ref{Hadamard}, $x\odot y$ is positive definite for any $x,y\in\mathbb{H}_n$. Now, let $v\in \mathbb{C}^{n}$ be the vector whose all components are equal to 1. Since ${\rm Tr} \ \{xy^{*}\}={\rm Tr} \ \{xy\}=v^{*}(x\odot y)v$, for any $x,y\in\mathbb{H}_n$, it follows that $0<v^{*}(x\odot y)v={\rm Tr} \ \{xy^{*}\}$. Second, we claim that $\sigma$ also attends Axiom ($A_2$). Indeed, let $x\in\mathbb{C}^{n\times n}$. Set $y\in\overline{\mathbb{H}_{n}}$ ($y\not={\bf 0}$). By Lemma \ref{Gram}, $y_{\iota\tilde{\iota}}=\langle u_{\iota},u_{\tilde{\iota}}\rangle$ ($\iota,\tilde{\iota}=1,\cdots,n$), for any set of vectors $\{u_{\iota}\in\mathbb{C}^{n}: \iota=1,\cdots, n\}$, not all zero. So, 
$$
\begin{array}{l}
0<{\rm Tr} \ \{xy^{*}\}={\rm Tr} \ \{xy\}=\displaystyle\sum_{\iota,\tilde{\iota}=1}^{n}x_{\iota\tilde{\iota}}y_{\iota\tilde{\iota}}\\
=\displaystyle\sum_{\iota,\tilde{\iota}=1}^{n}x_{\iota\tilde{\iota}} u_{\iota}^{*}u_{\tilde{\iota}}=u_{\iota}^{*}xu_{\tilde{\iota}}.
\end{array}
$$
Since  $\{u_{\iota}\in\mathbb{C}^{n}: \iota=1,\cdots, n\}$ is arbitrarily chosen once so is $y$, we conclude that $x$ is positive definite. By Lemma \ref{Cholesky}, $x\in\mathbb{H}_{n}$. Therefore, the proof is complete.
\end{proof}

Similarly, we have that $\mathbb{P}_{n}$ is also a domain of positivity with respect to $\sigma(x,y)={\rm Tr} \ \{xy^{T}\}$ (the proof is analogous to the precedent case). 
\begin{definicao}
A  linear operator $T:\mathbb{E}\to \mathbb{E}$ is said to be an automorphism on $D$ if $T(\overline{D})=\overline{D}$. 
\end{definicao}
\begin{definicao}
A group $G$ is said to be of Lie type if it is a $C^{\infty}$ manifold such that the map $G\times G\to G$ defined by $(c,\bar{c})\mapsto c\bar{c}^{-1}$ is of class $C^{\infty}$. 
\end{definicao}
Equivalently, $T$ is an automorphism on $D$ if and only if it is injective and $T(D)=D$. The set of all automorphisms on $D$, denoted by $\Sigma(D)$, is a group of Lie type.
\begin{definicao}
A function $N:\overline{D}\to \mathbb{E}$ is said to be a norm if the following axioms are fulfilled
\begin{itemize}
\item[($N_1$)] $N$ is real-analytic, positive on $D$ and continuous on $\overline{D}$; 
\item[($N_2$)] $N(a)=0$ if $a\in\partial D$, where $\partial D$  is the frontier of $D$; 
\item[($N_3$)] $N(T(a))=|{\rm det}({\cal T})|N(a)$, for every automorphism $T$ on $D$ and $a\in D$, where ${\cal T}$ is the matrix associated to $T$ and ${\rm det}({\cal T})$ is the determinant of ${\cal T}$.
\end{itemize}
\end{definicao}

\begin{definicao}
A domain of positivity $D$ is called homogeneous if the automorphisms on $D$ are transitive, i.e.,  for any two points $a_1,a_2\in D$, there exists an automorphism $T\in\Sigma(D)$ such that $T(a_1)=a_2$. 
\end{definicao}

Norms in a homogeneous domain of positivity differ by constants. It enables to think a norm as being unique in certain sense. Another property of norms is its continuity on the frontier of $D$, which is $N(a)$ approaches to zero as $a$ approaches to the finite portion of $\partial D$. Any loss of generality is produced in assuming that $D$ enjoys of a formal real Jordan algebra with unit element $e$ since only homogenous domains of positivity are considered in this work. See \cite{Barut} for building of formal real Jordan algebras in homogenous domains of positivity.

Let $N:\overline{D}\to \mathbb{E}$ be a norm. According to \cite[p. 191]{Rothaus}, by setting $G(a)=-{\rm log} \ N(a)$, $G''(a)$ is a symmetric positive definite bi-linear form on $T_{a}D$,  where $T_{a}D$ denotes the tangent space to $D$, at $a$, and it naturally induces a Riemannian structure in $D$. Namely, for any $a\in D$, the Riemannian metric on the tangent space to $D$, at $a$, represented by $\langle,\rangle_{a}$, is given by $$\langle p_1,p_2\rangle_{a}=\langle G''(a)p_1,p_2\rangle,$$ for any $p_1,p_2\in T_{a}D$.  What is more, all automorphism on $D$ are also isometries, with respect to this metric.  For instance, $\mathbb{P}_{n}$ and $\mathbb{H}_{n}$ are homogeneous domains of positivity and they can be provided with the Riemannian metric defined by the Hessian of the standard logarithmic barrier, given by $$F(x)=-{\rm log \ det}(x),$$          
since $N(x)={\rm det}(x)$ is a norm for both sets. Indeed, for any $n$-by-$n$ Hermitian (symmetric) matrices $y$ and $z$, the Riemannian metric defined by $F''$, at $x\in\mathbb{H}_{n} \ (\mathbb{P}_{n})$, is given by $$\langle y,z\rangle_{x}=\langle F''(x)y,z\rangle={\rm Tr}\{x^{-1}yx^{-1}z\}.$$ Furthermore, the Riemannian metric which was previously introduced makes the sectional curvature of $D$ be nonpositive everywhere. See Corollary 5.10 in \cite[p. 220]{Rothaus}. This implies that $D$ is a Riemannian manifold of Hadamard type. Cartan-Hadamard theorem assures that the exponential map is locally a diffeomorphism from the tangent plane onto the manifold in this case. As consequence, geodesics are uniquely determined and the Riemannian distance $d(a_1,a_2)$ between any two points $a_1,a_2\in D$ is achieved to the length of the geodesic segment $\gamma_{a_1a_2}$ that connects them.

\begin{definicao}[Specially reducible homogeneous domains of positivity]
A homogeneous domain of positivity $D$ is said to be specially reducible if there are a subdomain $D_1$ of $D$, a compact Lie group $D_2\subset \mathbb{E}$ and a $C^{\infty}$ map $\varphi:D\times D_2\to D$ that satisfies the following statements
\begin{itemize}
\item[($R_1$)] $\varphi(a,c_1c_2)=\varphi(\varphi(a,c_2),c_1)$ and $\varphi(a,\bar{e})=a$, for every $a\in D$ and $c_1,c_2\in D_2$, where $\bar{e}$ is the unit element of $D_2$;
\item[($R_2$)] $\varphi_{c}:D\to D$, defined by $\varphi_{c}(a)=\varphi(a,c)$, is an automorphism satisfying $\varphi_c(e)=e$, for any $c\in D_2$;
\item[($R_3$)] for every $a_1,a_2\in D$, there exists $c\in D_2$ for which $\varphi(a_1,c)=a_2$;
\item[($R_4$)] $\varphi:D_1\times D_2\to D$ is onto. 
\end{itemize} 
\end{definicao}

The term \textit{subdomain} in the definition above means that $D_1\subset D$ and itself is a homogenous domain of positivity, with respect to $\sigma$. Even the  Axioms ($R_1$), ($R_2$) and ($R_3$) indicate that  $D_2$ acts transitively on $D$. We also admit that either $D_2$ is itself connected or its isotropy group, which is the maximal compact subgroup of $D_2$, so is. A large discussion about Lie groups and Lie transformation group acting transitively on $C^{\infty}$ manifolds is presented in \cite{Sakai}. Taking into account Theorem \ref{Rham}, $D_1$ can be seen as the Cartesian product of a countable number of irreducible simply connected symmetric spaces (specifically, homogeneous domains of positivity) and $D_2$ is a Euclidean space that enjoys of a Lie structure. 

Any geodesic $\gamma$ in $D$ can be rewritten in terms of geodesics in $D_1$ and $D_2$ as $\gamma\equiv\xi\times\alpha$. Based on the Axiom ($R_4$) we define the following relation on $D_1\times D_2$
\begin{definicao}
Let $(b_1,c_1),(b_2,c_2)\in D_1\times D_2$. $(b_1,c_1)$ is said to be associated to $(b_2,c_2)$ through the relation $R_{\varphi}$ if $\varphi(b_1,c_1)=\varphi(b_2,c_2)$, i.e., $$(b_1,c_1)R_{\varphi}(b_2,c_2)\Leftrightarrow \varphi(b_1,c_1)=\varphi(b_2,c_2).$$
\end{definicao}

A single verification shows that $R_{\varphi}$ is reflexive, symmetric and transitive. Consequently, $R_{\varphi}$ is an equivalence relation on $D_1\times D_2$. Moreover, if we take into account the set $D_1\times D_2/R\varphi$, which is the quotient of $D_1\times D_2$ with respect to $R_{\varphi}$, then the mapping $\psi:D_1\times D_2/R\varphi\to D$ defined by $\psi([(b,c)])=\varphi(b,c)$ is one-to-one $C^{\infty}$ map since $\varphi$ is of class $C^{\infty}$. For instance, by putting $$D_1\equiv\mathbb{D}_{n}, D_2\equiv\mathbb{U}_{n}(\mathbb{O}_{n}), \varphi(\lambda,w)=w\lambda w^{*} (w\lambda w^T),$$ where $\mathbb{U}_{n}$ denotes the connected compact Lie group of unitary matrices, both sets $\mathbb{H}_{n}$ and $\mathbb{P}_{n}$ are specially reducible homogeneous domains of positivity. See \cite{Abrudan} for a brief on the Riemannian structure and optimization algorithms for $\mathbb{U}_{n}$.  

Let $\langle,\rangle_{c}$ be the Riemannian metric on $D_2$, for any $c\in D_2$. According to classical textbooks on Riemannian Geometry as \cite{Sakai} for example, $D_1\times D_2$ is a Riemannian manifold with respect to the product Riemannian metric $\langle,\rangle_{(b,c)}$, given by $$\langle (q_1,r_1),(q_2,r_2)\rangle_{(b,c)}=\langle q_1,q_2\rangle_{b}+\langle r_1,r_2\rangle_{c},$$ where $(q_1,r_1),(q_2,r_2)\in T_{(b,c)}D_1\times D_2$. In this case, it is used to assume the identification $T_{(b,c)}D_1\times D_2\cong T_{b}D_1\times T_{c}D_2$. On the other hand, the Riemannian structure of $D$ induces a Riemannian structure on $D_1\times D_2/R_{\varphi}$ since they are diffeomorphic, by  recalling that $\psi$ is a diffeomorphism of course.  Indeed, taking into account the identification $T_{[(b,c)]}D_1\times D_2/R_{\varphi}\cong T_{\varphi(b,c)}D$, a Riemannian metric is easily introduced on $D_1\times D_2/R_{\varphi}$ by putting $$\langle [(q_1,r_1)],[(q_2,r_2)]\rangle_{[(b,c)]}=\langle p_1,p_2\rangle_{\varphi(b,c)},$$ where $[(q_1,r_1)],[(q_2,r_2)]\in T_{[(b,c)]}D_1\times D_2/R_{\varphi}$ are respectively associated to $p_1,p_2\in T_{\varphi(b,c)}D$.
   
\subsection{Optimization problem and exact proximal point algorithm}
Let $D$ be a specially reducible homogeneous domain of positivity and $f:\overline{D}\to\mathbb{R}$ be convex on $D$. Assume that similar hypotheses to those used in \cite{Gregorio} are fulfilled, which are:
\begin{itemize}
\item[(${\cal H}_1$)] the set of minimizers to $f$ in $\overline{D}$ is nonempty; 
\item[(${\cal H}_2$)] for any $a\in\partial D$, $\displaystyle\lim_{k\to +\infty}f(a_{k})=f(a)$, for every sequence $\{a_{k}\}\subset D$ by satisfying $\displaystyle\lim_{k\to +\infty}a_{k}=a$.
\end{itemize}
The main aim of the current paper is to develop a proximal technique based on the proximal point algorithm in \cite{Ferreira2} to approach 
$$\displaystyle\min_{a\in \overline{D}} f(a).$$

Particularly, as previously discussed, for a given sequence of positive real numbers  $\{\beta_{k}\}$ and any point $\overline{a}=a_{0}\in D$, the proximal point method in \cite{Ferreira2} generates a sequence $\{a_{k}\}\subset D$ defined by 
\begin{equation}
a_{k+1}={\rm arg}\displaystyle\min_{a\in D}\left\{f(a)+\frac{\beta_{k}}{2}d^2(a,a_{k})\right\} \ \ (k=0,1,\cdots). \label{iteracao_principal}
\end{equation}
By Lemma 4.2 in \cite[p. 266]{Ferreira2}, $a_{k+1}$ exists and it is uniquely characterized by  
\begin{equation}
\beta_{k}{\rm exp}^{-1}_{a_{k+1}}a_{k}\in\partial f(a_{k+1}), \label{relacao1}
\end{equation}
for any $k$ ($\ k=0,1,\cdots$). The membership relation is replaced by equality and $\partial f(a_{k+1})$ by ${\rm grad} \ f (a_{k+1})$ whether $f$ is differentiable. Also, by Theorem 6.1 in \cite[p. 269]{Ferreira2}, $\{f(a_{k})\}$ converges to $\displaystyle\inf_{a\in D}f(a)$ whether $\displaystyle\sum_{k=0}^{\infty}\frac{1}{\beta_{k}}=+\infty$. Besides, if $\displaystyle\inf_{a\in D}f(a)$ is achieved in $D$ then $\{a_{k}\}$ converges to any $a^{*}\in D$ for which $f(a^{*})=\displaystyle\inf_{a\in D}f(a)$.

Now, Let $T_{k}$ be an automorphism on $D$ that attends $T_{k}(a_{k})=e$. Since $T_{k}\in\Sigma(D)$, it follows that $T_{k}$ is also an isometry on $D$. So, we have that $$d^{2}(a,a_{k})=d^{2}(T_{k}(a),e).$$
Define $\phi_{k},\rho_{k}:D_1\times D_2\to\mathbb{R}$ by 
\begin{equation}
\phi_{k}(b,c)=f\left(T^{-1}_{k}(\varphi(b,c))\right) 
\end{equation} 
and
\begin{equation}
\rho_{k}(b,c)=d^2(\varphi(b,c),e), 
\end{equation}
respectively. Given $(b_{0},c_{0})\in D_1\times D_2$, the current algorithm generates two sequences $\{b_{j}\}\in D_1$ and $\{c_{j}\}\in D_2$ defined by the following iterations
\begin{eqnarray}
b_{j+1}={\rm arg}\min_{b\in D_1}\left\{\phi_{k}(b,c_{j})+\displaystyle\frac{\beta_{k}}{2}\rho_{k}(b,c_{j})\right\},\label{passo_espectral}\\
c_{j+1}\in{\rm arg}\min_{c\in D_2}\left\{\phi_{k}(b_{j+1},c)+\displaystyle\frac{\beta_{k}}{2}\rho_{k}(b_{j+1},c)\right\},\label{passo_unitario}
\end{eqnarray}
for which $T^{-1}_{k}(\varphi(b_{j},c_{j}))$ converges to $a_{k+1}$ as proven bellow. By now, we call the attention by the proximal scheme  in the Table \ref{tabela1} that merges the algorithm in \cite{Ferreira2} and ideas discussed above.
\begin{table}[htb!]
\centering
\begin{tabular}{|l|}
\hline
\textbf{input} $a_{0}\in D, \beta_{0}>0, \theta\in(0,1]$.\\
\hline\hline
\hspace{0.3cm} $k\leftarrow 0$\\
\hspace{0.3cm} \textbf{while} $0\not\in\partial f(a_{k})$ \textbf{do}\\
\hspace{0.6cm} \textbf{input} $(b_{0},c_{0})\in D_1\times D_2$;\\
\hspace{0.6cm} \textbf{set} $u_{0}=T^{-1}_{k}\left(\varphi(b_{0},c_{0})\right)$;\\
\hspace{0.6cm} $j\leftarrow 0$;\\
\hspace{0.6cm} \textbf{while} $\beta_{k}{\rm exp}^{-1}_{u_{j}}a_{k}\not\in\partial f(u_{j})$ \textbf{do}\\ 
\hspace{0.9cm} \textbf{compute} $b_{j+1}$ \textbf{as in}  (\ref{passo_espectral});\\
\hspace{0.9cm} \textbf{compute} $c_{j+1}$ \textbf{by satisfying} (\ref{passo_unitario});\\ 
\hspace{0.9cm} $u_{j+1}=T^{-1}_{k}\left(\varphi(b_{j+1},c_{j+1})\right)$;\\
\hspace{0.9cm} $j\leftarrow j+1$;\\
\hspace{0.6cm} \textbf {end}\\
\hspace{0.6cm} $a_{k+1}=u_{j}$;\\
\hspace{0.6cm} $\beta_{k+1}=\theta\cdot\beta_{k}$;\\
\hspace{0.6cm} $k\leftarrow k+1$;\\
\hspace{0.3cm} \textbf{end}\\
\textbf{end}\\
\hline
\end{tabular}
\caption{Exact proximal point ($EPP$) algorithm.}
\label{tabela1}
\end{table}
\newpage
\subsubsection{Well-posedness of $b_{j+1}$ and existence of $c_{j+1}$}
\begin{lema}
Let $h:D\to\mathbb{R}$ be (strictly) convex, $T\in\Sigma(D)$ and $c\in D_2$. Then, $g_{c}:D_1\to\mathbb{R}$ defined by $g_{c}(b)=f\left(T(\varphi_{c}(b))\right)$ is (strictly) convex. \label{f_convexa}
\end{lema}
\begin{proof}
In fact, let $\gamma_{b_1b_2}:[t_1,t_2]\to D_1$ be the geodesic segment connecting $b_1$ to $b_2$ ($\gamma_{b_1b_2}(t_1)=b_1$, $\gamma_{b_1b_2}(t_2)=b_2$ and $\gamma_{b_1b_2}(t)\in D_1$, for every $t\in(t_1,t_2)$) and any $b_1,b_2\in D_{1}$. Then $\xi:[t_1,t_2]\to D$ defined by $$\xi(t)=T(\varphi_{c}(\gamma_{b_1b_2}(t)))$$ is the geodesic segment connecting $T(\varphi_{c}(b_1))$ to $T(\varphi_{c}(b_2))$, since $\varphi_{c}$ is a diffeomorphism and $T$ is an isometry. Thus
$$
\begin{array}{l}
g_{c}(\gamma_{b_1b_2}((1-t)\cdot t_1+t\cdot t_2))\\
=h\left(T(\varphi_{c}(\gamma_{b_1b_2}((1-t)\cdot t_1+t\cdot t_2)))\right)\\
=h(\xi((1-t)\cdot t_1+t\cdot t_2))\leq (1-t)\cdot h(\xi(t_1))+t\cdot h(\xi(t_2))\\
=(1-t)\cdot h\left(T(\varphi_{c}(\gamma_{b_1b_2}(t_1)))\right)+t\cdot h\left(T(\varphi_{c}(\gamma_{b_1b_2}(t_2)))\right)\\
= (1-t)\cdot g_{c}(\gamma_{b_1b_2}(t_1))+t\cdot g_{c}(\gamma_{b_1b_2}(t_2)),
\end{array}
$$
for every $t\in[0,1]$. Furthermore, if $f$ is strictly convex then the inequality is strict for every $t\in(0,1)$. So, the proof is complete.
\end{proof}

\begin{proposicao}
Let $f:\overline{D}\to\mathbb{R}$ be a convex function in $D$. Then, $\phi_{k}(\cdot,c)$ is convex in $D_1$, for any $c\in D_2$. \label{p_1} 
\end{proposicao}
\begin{proof}
It immediately follows from Lemma \ref{f_convexa} since $T_{k}^{-1}\in\Sigma(D)$ and $f\displaystyle\mid_{D}$ is convex.  
\end{proof}

\begin{proposicao}
$\rho_{k}(\cdot,c)$ is a $C^{\infty}$ strictly convex function in $D_1$, for any $c\in D_2$, and its gradient is given by $${\rm grad}_{b} \ \rho_{k}(b,c)=-2 \  exp^{-1}_{b}e.$$ \label{p_2}
\end{proposicao}
\begin{proof}
In fact, statements (1) and (2) of Theorem 4.1 in \cite[p. 11]{Bishop} affirms that $d^2_{\bar{a}}:D\to\mathbb{R}$ defined by $d^2_{\bar{a}}(a)=d^2(a,\bar{a})$ is a $C^{\infty}$ strictly convex function, for any $\bar{a}\in D$ once $D$ is of Hadamard type. Since $\varphi(b,c)$ is a diffeomorphism, for any $c\in D_2$, $\rho_{k}(\cdot,c)$ is a composition of a $C^{\infty}$ function with a $C^{\infty}$ map and therefore it is $C^{\infty}$ too. On the other hand, Proposition 4.8 in \cite[p. 108]{Sakai} states that ${\rm grad} \ d_{\bar{a}}(a)=\dot{\gamma}_{\bar{a}a}(d_{\bar{a}}(a))$, where $d_{\bar{a}}(a)=d(a,\bar{a})$ and $\dot{\gamma}_{\bar{a}a}(d_{\bar{a}}(a))$ is the velocity of the geodesic segment $\gamma_{\bar{a}a}$ in its end point, i.e., $\dot{\gamma}_{\bar{a}a}(d_{\bar{a}}(a))=exp^{-1}_{\bar{a}}a$. A single reparametrization of $\gamma_{\bar{a}a}$ as $\tilde{\gamma}_{\bar{a}a}(t)=\gamma_{\bar{a}a}(-t)$ enables to write $\dot{\gamma}_{\bar{a}a}(d_{\bar{a}}(a))=-exp^{-1}_{a}\bar{a}$. Thus, the product rule to derivatives in Riemannian manifolds establishes that 
\begin{equation}
{\rm grad} \ d^2_{\bar{a}}(a)=-2exp^{-1}_{a}\bar{a}. \label{aux}
\end{equation}
Now, by Axiom ($R_2$), $\rho_{k}(b,c)=d^2(\varphi(b,c),e)=d^2(\varphi_{c}(b),\varphi_{c}(e))=d^2(b,e)$. It means that $\rho_{k}(\cdot,c)$ does not depend on $c$. So, the strictly convexity of $\rho_{k}(\cdot,c)$  follows from Lemma \ref{f_convexa} and its gradient from (\ref{aux}).   
\end{proof}

We emphasize that the sum of a convex function with a strictly convex one become strictly convex. In addition, minimizers are uniquely determined whether they exist of course. Recall that the same occurs with convex functions in Euclidean spaces. 
\begin{proposicao}
There is only one $b_{j+1}$, for any $j$ ($j=0,1,\cdots$), and it is characterized by $$\beta_{k}exp^{-1}_{b_{j+1}}e\in \partial \phi_{k}(b_{j+1},c_{j}).$$ \label{prox_caracter}
\end{proposicao} 
\begin{proof}
Propositions \ref{p_1} and \ref{p_2} establish that $\left(\phi_{k}+\displaystyle\frac{\beta_{k}}{2}\rho_{k}\right)(\cdot,c_{j})$ is strictly convex in $D_1$.  On the other hand, an argumentation similar to the proof of Lemma 3 in \cite[p. 473]{Gregorio}, with $I$ replaced by $e$, shows that $\left(\phi_{k}+\displaystyle\frac{\beta_{k}}{2}\rho_{k}\right)(\cdot,c_{j})$ is 1-coercive. Therefore, the result follows.   
\end{proof}

We finish this section claiming the famous Weierstrass's Theorem to ensure the existence of $c_{j+1}$. Recall that $D_2$ is compact and  
$\left(\phi_{k}+\displaystyle\frac{\beta_{k}}{2}\rho_{k}\right)(b_{j+1},\cdot)$ is continuous, for every $j$ ($j=0,1,\cdots$).

\subsubsection{Convergence of $\displaystyle\left\{u_{j}\right\}$ to $a_{k+1}$}

\begin{proposicao}
Let $\{b_{j}\}$ and $\{c_{j}\}$ be the sequences generated by the iterations (\ref{passo_espectral}) and (\ref{passo_unitario}) respectively. Also, let $\displaystyle\left\{u_{j}\right\}$ be the inner sequence generated by the ($EPP$) algorithm in its $k$th iteration ($k=0,1,\cdots$), that is $u_{j}=T^{-1}_{k}\left(\varphi(b_{j},c_{j})\right)$ ($j=0,1,\cdots$). Define $\tilde{u}_{j+1}=T^{-1}_{k}\left(\varphi(b_{j+1},c_{j})\right)$ ($j=0,1,\cdots$). Then $$\phi_{k}(b_{j},c_{j})+\displaystyle\frac{\beta_{k}}{2}\rho_{k}(b_{j},c_{j})\geq \phi_{k}(b_{j+1},c_{j+1})+\displaystyle\frac{\beta_{k}}{2}\rho_{k}(b_{j+1},c_{j+1})+\frac{1}{2}d^2(u_{j},\tilde{u}_{j+1}),$$ for every $j\in\mathbb{N}$.\label{proposicaomonotona}
\end{proposicao}
\begin{proof}
Set the geodesic triangle $\triangle(b_{j}, b_{j+1},e)$ in $D_1$. Define by $\theta$ the angle between $exp^{-1}_{b_{j+1}}b_{j}$ and $exp^{-1}_{b_{j+1}}e$.  By the law of cosines in Hadamard manifolds (for instance, see Proposition 4.5 in \cite[p. 223]{Sakai}),
$$d^2(b_{j},e)\geq d^2(b_{j+1},b_{j})+d^2(b_{j+1},e)-2\cdot d(b_{j+1},b_{j})\cdot d(b_{j+1},e)\cdot{\rm cos} \ \theta.$$
Since $\langle exp^{-1}_{b_{j+1}}b_{j},exp^{-1}_{b_{j+1}}e\rangle_{b_{j+1}}=d(b_{j+1},b_{j})\cdot d(b_{j+1},e)\cdot {\rm cos} \ \theta,$
\begin{equation}
d^2(b_{j},e)\geq d^2(b_{j+1},b_{j})+d^2(b_{j+1},e)-2\cdot\langle exp^{-1}_{b_{j+1}}b_{j},exp^{-1}_{b_{j+1}}e\rangle_{b_{j+1}}.\label{desigualdade1}
\end{equation}
On the other hand, the convexity of $\phi_{k}(\cdot,c_{j})$ combined with Proposition \ref{prox_caracter}  implies that 
$$\phi_{k}(b,c_{j})\geq \phi_{k}(b_{j+1},c_{j})+\langle exp^{-1}_{b_{j+1}}b,\beta_{k}\cdot exp^{-1}_{b_{j+1}}e \rangle_{b_{j+1}},$$
for any $b\in D_1$. In particular, 
\begin{equation}
\frac{1}{\beta_{k}}\left[\phi_{k}(b_{j},c_{j})-\phi_{k}(b_{j+1},c_{j})\right]\geq\langle exp^{1}_{b_{j+1}}b_{j},exp^{1}_{b_{j+1}}e \rangle_{b_{j+1}}.\label{desigualdade2}
\end{equation}
 It follows from inequalities (\ref{desigualdade1}) and (\ref{desigualdade2}) that  
$$d^2(b_{j},e)\geq d^2(b_{j+1},b_{j})+d^2(b_{j+1},e)-\frac{2}{\beta_{k}}\left[\phi_{k}(b_{j},c_{j})-\phi_{k}(b_{j+1},c_{j})\right].$$
Multiplying the inequality above by $\frac{\beta_{k}}{2}$ and reorganizing the terms, we have that 
$$\phi_{k}(b_{j+1},c_{j})+\frac{\beta_{k}}{2}d^2(b_{j+1},e)+\frac{\beta_{k}}{2}d^2(b_{j},b_{j+1})\leq \phi_{k}(b_{j},c_{j})+\frac{\beta_{k}}{2}d^2(b_{j},e).$$
Recalling that $\rho_{k}(b,c)=d^2(\varphi(b,c),e)=d^2(\varphi_{c}(b),\varphi_{c}(e))=d^2(b,e)$, it follows that
$$\phi_{k}(b_{j+1},c_{j})+\frac{\beta_{k}}{2}\rho_{k}(b_{j+1},c_{j})+\frac{\beta_{k}}{2}d^2(b_{j},b_{j+1})\leq \phi_{k}(b_{j},c_{j})+\frac{\beta_{k}}{2}\rho_{k}(b_{j},c_{j}).$$
To finishing our argumentation $$\phi_{k}(b_{j+1},c_{j+1})+\frac{\beta_{k}}{2}\rho_{k}(b_{j+1},c_{j+1})\leq \phi_{k}(b_{j+1},c_{j})+\frac{\beta_{k}}{2}\rho_{k}(b_{j+1},c_{j}),$$ by (\ref{passo_unitario}), and $T^{-1}_{k}\circ\varphi_{c_{j}}\in\Sigma(D)$ since $T^{-1}_{k}$ and $\varphi_{c_{j}}$ are automorphism on $D$. It still results from the fact of $T^{-1}_{k}\circ\varphi_{c_{j}}$ be an automorphism that  $$d^2(b_{j},b_{j+1})=d^2\left(T^{-1}_{k}\circ\varphi_{c_{j}}(b_{j}),T^{-1}_{k}\circ\varphi_{c_{j}}(b_{j+1})\right)=d^2(u_{j},\tilde{u}_{j+1}).$$ Therefore, the proposition follows from the last two statement above.
\end{proof}
\begin{corolario}
If $b_{j+1}\not=b_{j}$ then
$$\phi_{k}(b_{j},c_{j})+\frac{\beta_{k}}{2}\rho_{k}(b_{j},c_{j})> \phi_{k}(b_{j+1},c_{j+1})+\frac{\beta_{k}}{2}\rho_{k}(b_{j+1},c_{j+1}).$$\label{corolario1}
\end{corolario}
\begin{proof}
Since that $b_{j+1}\not=b_{j}$, $0<d(b_{j},b_{j+1})=d(u_{j},\tilde{u}_{j+1})$. So, the statement results from Proposition \ref{proposicaomonotona}. 
\end{proof} 
\begin{corolario}
Suppose that the hypothesis of Corollary \ref{corolario1} fails. If $\phi_{k}(b_{j},c_{j+1})<\phi_{k}(b_{j},c_{j})$ then its statement is still assured.\label{corolario2}
\end{corolario}
\begin{proof}
Notice that, under the negation of the hypothesis of Corollary \ref{corolario1}, $b_{j+1}=b_{j}$. It implies that $\tilde{u}_{j+1}=u_{j}$ and, consequently, that $d^2(u_{j},\tilde{u}_{j+1})=0$. However, after replacing the current hypothesis at Proposition \ref{proposicaomonotona}, the result still follows.
\end{proof}
\begin{proposicao}[Inner stopping criteria]
Let $\{b_{j}\}$ and $\{c_{j}\}$ be the sequences generated by (\ref{passo_espectral}) and (\ref{passo_unitario}), respectively. If the following statements are fulfilled,  
\begin{itemize}
\item[($i$)] $b_{j+1}=b_{j}$,
\item[($ii$)] $c_{j}\in{\rm arg}\displaystyle\min_{c\in D_2}\left\{\phi_{k}(b_{j+1},c)+\displaystyle\frac{\beta_{k}}{2}\rho_{k}(b_{j+1},c)\right\},$
\end{itemize}
then $u_{j}=a_{k+1}$.\label{complexo}
\end{proposicao}
\begin{proof}
In fact, suppose that $u_{j}\not=a_{k+1}$. Then, for any $\delta>0$, there exists $a_{\delta}\in B_{\delta}(u_{j})=\{a\in D: d(a,u_{j})<\delta\}$ for which $$f(a_{\delta})+\frac{\beta_{k}}{2}d^2(a_{\delta}, a_{k})<f(u_{j})+\frac{\beta_{k}}{2}d^2(u_{j}, a_{k}),$$ since $\left(f+\frac{\beta_{k}}{2}d^2_{a_{k}}\right)$ is strictly convex. By putting $a_{\delta}=T^{-1}_{k}(\varphi(b_{\delta},c_{\delta}))$\footnote{Notice that it is always possible to rewrite $a_{\delta}$ of that form since $T^{-1}_{k}$ is an automorphism. The existence of $b_{\delta}$ and $c_{\delta}$ is assured by Axiom ($R_4$).}, for any $b_{\delta}\in D_1$ and $c_{\delta}\in D_2$, it is the same to affirm that 
$$\phi_{k}(b_{\delta},c_{\delta})+\frac{\beta_{k}}{2}\rho_{k}(b_{\delta},c_{\delta})<\phi_{k}(b_{j},c_{j})+\frac{\beta_{k}}{2}\rho_{k}(b_{j},c_{j}).$$
Here, we detach two trivial cases: or ${\bf b_{\delta}=b_{j}}$ and ${\bf c_{\delta}\not=c_{j}}$ or ${\bf b_{\delta}\not=b_{j}}$ and ${\bf c_{\delta}=c_{j}}$. The first case contradicts the hypothesis ($ii$) for $\delta$ small enough. On the other hand, the hypothesis ($i$) is corrupted in the second case. Thus, to continuing from this point the hard case is assumed, i. e., ${\bf b_{\delta}\not=b_{j}}$ and ${\bf c_{\delta}\not=c_{j}}$. Let $\gamma$ be the geodesic segment in $D$ by connecting $a_{\delta}$ to $u_{j}$, i.e., $\gamma(0)=a_{\delta}$, $\gamma(1)=u_{j}$ and $\gamma(t)\in D$, for every $t\in(0,1)$. Then, by Rham Decomposition theorem for Simply connected symmetric spaces, there exist geodesic segments $\xi$ in $D_1$ and $\alpha$ in $D_2$ by connecting $b_{\delta}$ to $b_{j}$ and $c_{\delta}$ to $c_{j}$, respectively, for which $\gamma(t)=T^{-1}_{k}(\varphi(\xi(t),\alpha(t)))$, for any $t\in[0,1]$. Furthermore, the strictly convexity of $\left(\phi_{k}+\frac{\beta_{k}}{2}\rho_{k}\right)$ guarantees that $$\phi_{k}(\xi(t),\alpha(t))+\frac{\beta_{k}}{2}\rho_{k}(\xi(t),\alpha(t))<\phi_{k}(b_{j},c_{j})+\frac{\beta_{k}}{2}\rho_{k}(b_{j},c_{j}),$$ for any $t\in[0,1)$. Let $\{t_l\}\subset \mathbb{R}_{+}$, where $\mathbb{R}_{+}=\{t\in\mathbb{R}:t\geq 0\}$, $\{b_{l}\}\subset D_1$ and $\{c_{l}\}\subset D_2$ be sequences built as follows: ($a$)  set $t_{0}=0$, $b_{0}=b_{\delta}$, $c_{0}=c_{\delta}$ and $l=0$; ($b$) compute the largest value $t_{l+1}\in [t_{l},1]$ for which $$\phi_{k}(b_{l},\alpha(t_{l+1}))+\frac{\beta_{k}}{2}\rho_{k}(b_{l},\alpha(t_{l+1}))\leq\phi_{k}(b_{j},c_{j})+\frac{\beta_{k}}{2}\rho_{k}(b_{j},c_{j});$$ ($c$) define $c_{l+1}=\alpha(t_{l+1})$, $b_{l+1}=\xi_{l}\left(\frac{1}{2}\right)$, where $\xi_{l}$ is the geodesic segment by connecting $b_{l}$ to $\xi(t_{l+1})$ ; ($d$) update $l=l+1$ and return to ($b$). By construction, $\{t_l\}$ is monotone nondecreasing and bounded above by 1. Consequently, it converges to any $\tilde{t}\in(0,1]$. Also by construction, $c_{l}$ converges to $\alpha(t_{\tilde{l}})$ and $\xi(t_{l+1})$ to $\xi(t_{\tilde{l}})$. According to \cite[p. 113]{Lim},  domains of positivity are particular cases of Bruhat-Tits space where the Semiparallelogram Law is attended at the midpoint of the geodesic segment which connects any pair of their points. So,  $$d^2(b_{l+1},a)\leq\frac{d^2(b_{l},a)+d^2(\xi(t_{l+1}),a)}{2}-\frac{d^2(b_{l},\xi(t_{l+1}))}{4}\leq \frac{d^2(b_{l},a)+d^2(\xi(t_{l+1}),a)}{2},$$ for every $a\in D$. In particular, for $a=\xi(t_{\tilde{l}})$, $$d^2(b_{l+1},\xi(t_{\tilde{l}})\leq\frac{d^2(b_{l},\xi(t_{\tilde{l}}))+d^2(\xi(t_{l+1}),\xi(t_{\tilde{l}}))}{2}.$$ However $d(\xi(t_{l+1}),\xi(t_{\tilde{l}}))$ is closed to zero for $l$ great enough. This implies that exists $l_0\in\mathbb{N}$, sufficiently great, for which $d^2(b_{l+1},\xi(t_{\tilde{l}}))\leq d^2(b_{l},\xi(t_{\tilde{l}}))$, for every $l\geq l_0$. As consequence, it follows that $d(b_{l+1},\xi(t_{\tilde{l}}))\leq d(b_{l},\xi(t_{\tilde{l}}))$, for every $l\geq l_0$. So, we conclude that $\{b_{l}\}$ belongs in a normal ball centered in $\xi(t_{\tilde{l}})$, with radius $r=\max_{0\leq l\leq l_0}d(b_{l},\xi(t_{\tilde{l}}))$. It means that $\{b_{l}\}$ is bounded. On the other hand, we claim that $$\phi_{k}(b_{l},c_{l})+\frac{\beta_{k}}{2}\rho_{k}(b_{l},c_{l})<\phi_{k}(b_{j},c_{j})+\frac{\beta_{k}}{2}\rho_{k}(b_{j},c_{j}),$$
for every $l\in \mathbb{N}\cup\{0\}$. By induction over $l$, the initial hypothesis, that is $$\phi_{k}(b_{0},c_{0})+\frac{\beta_{k}}{2}\rho_{k}(b_{0},c_{0})<\phi_{k}(b_{j},c_{j})+\frac{\beta_{k}}{2}\rho_{k}(b_{j},c_{j}),$$ is assured by definition of $b_{0}$ and $c_{0}$. Now, Suppose that the induction hypothesis is fulfilled. By definition of $c_{l+1}$, $$\phi_{k}(b_{l},c_{l+1})+\frac{\beta_{k}}{2}\rho_{k}(b_{l},c_{l+1})\leq\phi_{k}(b_{j},c_{j})+\frac{\beta_{k}}{2}\rho_{k}(b_{j},c_{j}).$$
Taking into account that $T^{-1}_{k}\left(\varphi(\xi_{l}(t),c_{l+1})\right)$ is the geodesic segment in $D$ which connects\\ $T^{-1}_{k}\left(\varphi(b_{l},c_{l+1})\right)$ to  $T^{-1}_{k}\left(\varphi(\xi(t_{l+1}),c_{l+1})\right)$ and the strict convexity of $\left(\phi_{k}+\frac{\beta_{k}}{2}\rho_{k}\right)$, 
$$
\begin{array}{l}
\phi_{k}(\xi_{l}(t),c_{l+1})+\frac{\beta_{k}}{2}\rho_{k}(\xi_{l}(t),c_{l+1})\\
<(1-t)\cdot\left[\phi_{k}(b_{l},c_{l+1})+\frac{\beta_{k}}{2}\rho_{k}(b_{l},c_{l+1})\right]+t\cdot\left[\phi_{k}(\xi(t_{l+1}),c_{l+1})+\frac{\beta_{k}}{2}\rho_{k}(\xi(t_{l+1}),c_{l+1})\right]\\
=(1-t)\cdot\left[\phi_{k}(b_{l},c_{l+1})+\frac{\beta_{k}}{2}\rho_{k}(b_{l},c_{l+1})\right]\\
+t\cdot\left[\phi_{k}(\xi(t_{l+1}),\alpha(t_{l+1}))+\frac{\beta_{k}}{2}\rho_{k}(\xi(t_{l+1}),\alpha(t_{l+1}))\right]\\
<(1-t)\cdot\left[\phi_{k}(b_{j},c_{j})+\frac{\beta_{k}}{2}\rho_{k}(b_{j},c_{j})\right]+t\cdot\left[\phi_{k}(b_{j},c_{j})+\frac{\beta_{k}}{2}\rho_{k}(b_{j},c_{j})\right]\\
=\phi_{k}(b_{j},c_{j})+\frac{\beta_{k}}{2}\rho_{k}(b_{j},c_{j}),
\end{array}
$$
for every $t\in (0,1)$. In particular, for $t=\frac{1}{2}$, the plea follows. Now, the proof is divided in two cases. First, suppose that $\tilde{t}=1$. In addition, admit that this is achieved in a finite number of steps, i. e., there exists $l_0\in\mathbb{N}$ such that $t_l=1$, for every $l\geq l_0$. Without loss of generality, assume that $l_0$ is the first index for which $t_{l_0}=1$. Again by construction, $$\phi_{k}(b_{(l_0-1)},c_{j})+\frac{\beta_{k}}{2}\rho_{k}(b_{(l_0-1)},c_{j})\leq\phi_{k}(b_{j},c_{j})+\frac{\beta_{k}}{2}\rho_{k}(b_{j},c_{j}),$$ with $b_{(l_0-1)}\not=b(t_{(l_0)})=b(1)=b_{j}$. This implies that $b_{(l_0)}=\xi_{(l_0)}(\frac{1}{2})\not=b_{j}$, where $\xi_{(l_0)}$ is the geodesic segment that connects $b_{(l_0-1)}$ to $b(t_{(l_0)})=b_{j}$. From this point, $c(t_{l})=c_{j}$ and $b_{l}=b_{(l_0)}$, for every $l>l_0$. This also corrupts the hypothesis ($i$) since 
$$
\begin{array}{l}
\phi_{k}(b_{(l_0)},c_{(l_0)})+\frac{\beta_{k}}{2}\rho_{k}(b_{(l_0)},c_{(l_0)})\\
=\phi_{k}(b_{(l_0)},c_{j})+\frac{\beta_{k}}{2}\rho_{k}(b_{(l_0)},c_{j})<\phi_{k}(b_{j},c_{j})+\frac{\beta_{k}}{2}\rho_{k}(b_{j},c_{j}).
\end{array}
$$
Assume that the convergence of $t_{l}$ to 1 is not finite.  Let $\{b_{(l_q)}\}$ be a convergent subsequence of $\{b_{l}\}$ and $\tilde{b}$ its limit point. The continuity of $(\phi_{k}+\frac{\beta_{k}}{2}\rho_{k})$ implies that  
$$
\begin{array}{l}
\phi_{k}(\tilde{b},c_{j})+\frac{\beta_{k}}{2}\rho_{k}(\tilde{b},c_{j})\\
=\lim_{q\to +\infty}\left[\phi_{k}(b_{(l_q)},c_{(l_q)})+\frac{\beta_{k}}{2}\rho_{k}(b_{(l_q)},c_{(l_q)})\right]
\leq\phi_{k}(b_{j},c_{j})+\frac{\beta_{k}}{2}\rho_{k}(b_{j},c_{j}).
\end{array}
$$
Again, by the continuity of $(\phi_{k}+\frac{\beta_{k}}{2}\rho_{k})$, there exists $q_0\in\mathbb{N}$ such that $$\phi_{k}(b_{(l_q)},c_{j})+\frac{\beta_{k}}{2}\rho_{k}(b_{(l_q)},c_{j})\leq\phi_{k}(b_{j},c_{j})+\frac{\beta_{k}}{2}\rho_{k}(b_{j},c_{j}),$$ for every $q\geq q_0$, with $b_{(l_{q_0})}\not=b_{j}$ since the convergence of $t_{l}$ to 1 is not achieved in a finite number of steps. Let $\tilde{\xi}$ be the geodesic segment that connects $\tilde{b}$ to $b_{j}$ whether $\tilde{b}\not=b_{j}$. Notice that if $\tilde{b}=b_{j}$ then we can replace $\tilde{b}$ by $b_{l_{q_0}}$. The strict convexity of $(\phi_{k}+\frac{\beta_{k}}{2}\rho_{k})$ implies that   
$$
\begin{array}{l}
\phi_{k}(\tilde{\xi}(\frac{1}{2}),c_{j})+\frac{\beta_{k}}{2}\rho_{k}(\tilde{\xi}(\frac{1}{2}),c_{j})\\
<\frac{1}{2}\left[\phi_{k}(\tilde{b},c_{j})+\frac{\beta_{k}}{2}\rho_{k}(\tilde{b},c_{j})\right]+\frac{1}{2}\left[\phi_{k}(b_{j},c_{j})+\frac{\beta_{k}}{2}\rho_{k}(b_{j},c_{j})\right]\\
\leq \frac{1}{2}\left[\phi_{k}(b_{j},c_{j})+\frac{\beta_{k}}{2}\rho_{k}(b_{j},c_{j})\right]+\frac{1}{2}\left[\phi_{k}(b_{j},c_{j})+\frac{\beta_{k}}{2}\rho_{k}(b_{j},c_{j})\right]\\
=\phi_{k}(b_{j},c_{j})+\frac{\beta_{k}}{2}\rho_{k}(b_{j},c_{j})
\end{array}
$$
and it still contradicts ($i$). We emphasize that $b_{l_{q_0}}\not=b_{j}$ because if this is not true then, by construction, $t_{l_q}$ would be constant, or strictly smaller than 1 or equal to 1, for every $q>q_0$. It would contradict the assumption or that $t_l$ converges to 1 or that the convergence of $t_l$ to 1 is not finite. By finishing our argumentation,  admit that $\tilde{t}\in(0,1)$. We call the attention to the fact that it only happens if $b_{l}$ is sufficiently closed to $b_{j}$, for $l$ great enough. Therefore,  
$$
\begin{array}{l}
\phi_{k}(b_{j},c_{\tilde{l}})+\frac{\beta_{k}}{2}\rho_{k}(b_{j},c_{\tilde{l}})\\
=\displaystyle\lim_{l\to+\infty}\left[\phi_{k}(b_{l},c_{l})+\frac{\beta_{k}}{2}\rho_{k}(b_{l},c_{l})\right]
\leq\phi_{k}(b_{j},c_{j})+\frac{\beta_{k}}{2}\rho_{k}(b_{j},c_{j}),
\end{array}
$$ 
where $c_{\tilde{l}}=\alpha(t_{\tilde{l}})\not=c_{j}$ since $t_{\tilde{l}}\not=1$. Analogously to the precedent case, $T^{-1}_{k}\left(\varphi(b_{j},\alpha(t)\right)$, for $t\in[t_{\tilde{l}},1]$, is the geodesic segment in $D$ which connects $T^{-1}_{k}\left(b_{j},c_{\tilde{l}}\right)$ to  $T^{-1}_{k}\left(b_{j},c_{j}\right)$\footnote{The restriction of a geodesic $\alpha$ in any complete Riemannian manifold to a closed subinterval $[a,b]\in\mathbb{R}$ become a geodesic segment by connecting $\alpha(a)$ to $\alpha(b)$.}. Again, by the strict convexity of $(\phi_{k}+\frac{\beta_{k}}{2}\rho_{k})$, we have
$$
\begin{array}{l}
\phi_{k}(b_{j},\alpha(t))+\frac{\beta_{k}}{2}\rho_{k}(b_{j},\alpha(t))\\
<(1-t)\cdot\left[\phi_{k}(b_{j},c_{\tilde{l}})+\frac{\beta_{k}}{2}\rho_{k}(b_{j},c_{\tilde{l}})\right]
+t\cdot\left[\phi_{k}(b_{j},c_{j})+\frac{\beta_{k}}{2}\rho_{k}(b_{j},c_{j})\right]\\
\leq (1-t)\cdot\left[\phi_{k}(b_{j},c_{j})+\frac{\beta_{k}}{2}\rho_{k}(b_{j},c_{j})\right]
+t\cdot\left[\phi_{k}(b_{j},c_{j})+\frac{\beta_{k}}{2}\rho_{k}(b_{j},c_{j})\right]\\
=\phi_{k}(b_{j},c_{j})+\frac{\beta_{k}}{2}\rho_{k}(b_{j},c_{j}),
\end{array}
$$ 
for every $t\in(t_{\tilde{l}},1)$. Therefore, for $t$ sufficiently close to $1$, the last inequality contradicts ($ii$).
\end{proof}
\begin{teorema}
The sequence $\{u_{j}\}_{j\in\mathbb{N}\cup{0}}$ defined by $u_{j}=T^{-1}_{k}\left(\varphi(b_{j},c_{j})\right), \ j=0,1,\cdots$, converges to $a_{k+1}$.
\end{teorema}
\begin{proof}
In fact, if $b_{j+1}=b_{j}$ and $c_{j}\in{\rm arg}\min_{c\in D_2}\left(\phi_{k}(b_{j+1},c)+\displaystyle\frac{\beta_{k}}{2}\rho_{k}(b_{j+1},c)\right)$ then $u_{j+1}$ is updated as the own iterated $u_{j}$. So, $u_{l}=u_{j}$, for every $l\geq j$, $l\in\mathbb{N}$, i.e., starting this point the sequence is constant. However, by Proposition \ref{complexo}, $u_{j}=a_{k+1}$. On the other hand, if  $b_{j+1}\not=b_{j}$ or $c_{j}\not\in{\rm arg}\min_{c\in D_2}\{\phi_{k}(b_{j+1},c)+\displaystyle\frac{\beta_{k}}{2}\rho_{k}(b_{j+1},c)\}$, for every $j=0,1,\cdots$, then, by Corollaries \ref{corolario1} and \ref{corolario2}, the sequence $\left\{\left(f+\frac{\beta_{k}}{2}d^2_{a_{k}}\right)(u_{j})\right\}$ is monotone decreasing. Therefore, it converges to $\left(f+\frac{\beta_{k}}{2}d^2_{a_{k}}\right)(a_{k+1})$. Since $a_{k+1}$ is uniquely determined, the result follows.
\end{proof}

\subsubsection{On the global convergence to the $EPP$ algorithm}

As cited at the beginning of the current section, $f(a_{k})\to \displaystyle\inf_{a\in D} f(a)$. Moreover, if the set  of minimizers to $f$ intersects $D$ then the sequence $\{a_{k}\}$ generated by the ($EPP$) Algorithm converges to any point from this intersection. On the other hand, if the intersection between the set of minimizers to $f$ and $D$ is empty then, by (${\cal H}_1$), all minimizers to $f$ belongs to $\partial D$. Thus, by (${\cal H}_2$) and a similar argumentation to that used in the proof of Lemma 6 in \cite[p. 475]{Gregorio}, the convergence of $f(a_{k})$ to $\displaystyle\min_{a\in\overline{D}}f(a)$ is fulfilled. So, it enables to enunciate a result analogous to Theorem 1 in \cite[p. 475]{Gregorio}.
\begin{teorema}[weak and strong convergence: exact version]
Let $f:\overline{D}\to \mathbb{R}$ and $\{a_{k}\}$ be respectively a convex function in $D$ and the sequence generated by the ($EPP$) algorithm. Then $a_{k}$ weakly converges to $a^{*}$, with respect to $f$, for any $a^{*}\in {\rm arg}\displaystyle\min_{a\in \overline{D}} f(a)$. Moreover, if ${\rm arg}\displaystyle\min_{a\in \overline{D}} f(a)\cap D\not=\emptyset$ then $a_{k}$ strongly converges to any $a^{*}\in{\rm arg}\displaystyle\min_{a\in \overline{D}} f(a)\cap D$. 
\end{teorema}
\begin{proof}
In fact, given $\beta_{0}>0$,  $\{\beta_{k}\}$ satisfies $\beta_{k}=\theta_{k}\beta_{0}$, for every $k=0,1,\cdots$, with $\theta\in(0,1]$. Hence, $\displaystyle\sum_{k=0}^{\infty}\frac{1}{\beta_{k}}$ is a divergent series. Since $D$ is a manifold of Hadamard type, by Theorem 6.1 in \cite{Ferreira2} and hypotheses (${\cal H}_1$) and (${\cal H}_2$) the statements follows.  
\end{proof}
\begin{observacao}
Notice that we can not assure the strong convergence of $\{a_{k}\}$ to $a^{*}$ whether $a^{*}\in\partial D$, no in a Riemannian way since $D$ is open. Indeed, there is no geodesic segment connecting any point from $D$ to $a^{*}$ in this case. However, the most important statement is fulfilled which is the convergence of $f(a_{k})$ to $\displaystyle\inf_{a\in D} f(a)=\displaystyle\min_{a\in\overline{D}}f(a)$. Particularly, the result above shows that the current method is of the same nature to the algorithm stated in \cite{Gregorio}, and, in certain sense, it can be simultaneously interpreted as an extension and a improvement of it, regarded the way to update $c_{j}$ at (\ref{passo_unitario}) in replacement of the way to update $w_{j}$ at ($S_3$). See the Exact $SDPProx$ algorithm at \cite[p. 473]{Gregorio} for the original description to the algorithm.   
\end{observacao}

\section{Notes on the inexact version}
In this section we discuss about features on the inexact version to the current method. Almost all statement has already been guaranteed to the particular case of the manifold of symmetric positive definite matrices in \cite{Gregorio}. We only emphasize that the biggest part of the argumentation used in proofs there become valid to homogeneous domains of positivity of specially reducible type.     
\subsection{On the technique}
Analogous to the assumption (16) in \cite[p. 475]{Gregorio}, we admit that $a_{k+1}$ is not determined exactly. In terms of (\ref{relacao1}), it is the same of assuming the following relation 
\begin{equation}
\beta_{k}{\rm exp}^{-1}_{a_{k+1}}a_{k}\in\partial_{\epsilon_{k}} f(a_{k+1}), \label{relacao2}
\end{equation} 
for any $\epsilon_{k}\geq 0$, $k=0,1,\cdots$, where $\partial_{\epsilon} f(\tilde{a})$ denotes the $\epsilon$-subdifferential of $f$ at $\tilde{a}\in D$, for any $\epsilon\geq 0$. Namely, $$\partial_{\epsilon} f(\tilde{a}):=\{s\in\mathbb{E}^{n}:f(a)\geq f(\tilde{a})+\langle s,{\rm exp}^{-1}_{\tilde{a}}a \rangle_{\tilde{a}}-\epsilon,\forall a\in D\}.$$
This assumption is weaker than (\ref{relacao1}) since $\partial_{\epsilon} f(\tilde{a})\supseteq \partial f(\tilde{a})$, for any $\tilde{a}\in D$ and $\epsilon\geq 0$. However, in a practical way, it enables to use a finite stopping criteria for each iteration of the proximal point method. The following scheme resumes the inexact version to the current algorithm
\begin{table}[htb!]
\centering
\begin{tabular}{|l|}
\hline
\textbf{input} $a_{0}\in D, \beta_{0}>0, \epsilon_{0}\geq 0, \theta_1\in(0,1],\theta_2\in(0,1)$.\\
\hline\hline
\hspace{0.3cm} $k\leftarrow 0$\\
\hspace{0.3cm} \textbf{while} $0\not\in\partial f(a_{k})$ \textbf{do}\\
\hspace{0.6cm} \textbf{input} $(b_{0},c_{0})\in D_1\times D_2$;\\
\hspace{0.6cm} \textbf{set} $u_{0}=T^{-1}_{k}\left(\varphi(b_{0},c_{0})\right)$;\\
\hspace{0.6cm} $j\leftarrow 0$;\\
\hspace{0.6cm} \textbf{while} $\beta_{k}{\rm exp}^{-1}_{u_{j}}a_{k}\not\in\partial_{\epsilon_{k}} f(u_{j})$ \textbf{do}\\ 
\hspace{0.9cm} \textbf{compute} $b_{j+1}$ \textbf{as in}  (\ref{passo_espectral});\\
\hspace{0.9cm} \textbf{compute} $c_{j+1}$ \textbf{by satisfying} (\ref{passo_unitario});\\ 
\hspace{0.9cm} $u_{j+1}=T^{-1}_{k}\left(\varphi(b_{j+1},c_{j+1})\right)$;\\
\hspace{0.9cm} $j\leftarrow j+1$;\\
\hspace{0.6cm} \textbf {end}\\
\hspace{0.6cm} $a_{k+1}=u_{j}$;\\
\hspace{0.6cm} $\beta_{k+1}=\theta_1\cdot\beta_{k}$;\\
\hspace{0.6cm} $\epsilon_{k+1}=\theta_2\cdot\epsilon_{k}$;\\
\hspace{0.6cm} $k\leftarrow k+1$;\\
\hspace{0.3cm} \textbf{end}\\
\textbf{end}\\
\hline
\end{tabular}
\caption{Inexact proximal point ($IPP$) algorithm.}
\label{tabela2}
\end{table}

\subsubsection{On the global convergence}

Without loss of generality, Lemma 7 in \cite[p. 476]{Gregorio} can be full imported to the current method as enunciated in the following. Its proof is omitted since the argumentation exactly follows the same steps described there.    
\begin{proposicao} 
Let $f:\overline{D}\to \mathbb{R}$ and $\{a_{k}\}$ be a convex function in $D$ and the sequence generated by the ($EPP$) algorithm respectively. If the relation (\ref{relacao2}) is satisfied then the following inequality is fulfilled
$$d^2(a_{k+1},a)\leq d^2(a_{k},a)-d^2(a_{k+1},a_{k})+\frac{2}{\beta_{k}}\left(f(a)-f(a_{k+1})\right)+2\frac{\epsilon_{k}}{\beta_{k}}.$$  \label{lema_especial}
\end{proposicao}
We call the attention to the fact that the result above is held good to manifolds of Hadamard type once its proof uses general features from this kind of manifold. The same happens with the result bellow. See Theorem 2 in \cite[p. 477]{Gregorio}.    
\begin{teorema}[weak and strong convergence: inexact version]
Let $f:\overline{D}\to \mathbb{R}$ and $\{a_{k}\}$ be  a convex function in $D$ and the sequence generated by the ($IPP$) Algorithm respectively. If $\displaystyle\frac{\theta_2}{\theta_1}<1$ then $a_{k}$ weakly converges to $a^{*}$ with respect to $f$, for any $a^{*}\in {\rm arg}\displaystyle\min_{a\in \overline{D}} f(a)$. Moreover, if ${\rm arg}\displaystyle\min_{a\in \overline{D}} f(a)\cap D\not=\emptyset$ then $a_{k}$ strongly converges to any $a^{*}\in{\rm arg}\displaystyle\min_{a\in \overline{D}} f(a)\cap D$. 
\label{convergencia_versao_inexata}
\end{teorema} 
\begin{proof}
In fact, given $\beta_{0},\epsilon_{0}>0$,  $\{\beta_{k}\}$ and $\{\epsilon_{k}\}$  satisfy $\beta_{k}=(\theta_1)_{k}\cdot\beta_{0}$ and $\epsilon_{k}=(\theta_2)_{k}\cdot\epsilon_{0}$, respectively, for every $k=0,1,\cdots$, with $\theta_1\in(0,1]$ and $\theta_2\in(0,1)$. Hence, $\displaystyle\sum_{k=0}^{\infty}\frac{1}{\beta_{k}}$ diverges and $\displaystyle\sum_{k=0}^{\infty}\epsilon_{k}$ as well as $\displaystyle\sum_{k=0}^{\infty}\frac{\epsilon_{k}}{\beta_{k}}$ converges. From this point, the argumentation strongly uses Proposition \ref{lema_especial} and the same steps from the proof of the Theorem 2 in \cite{Gregorio}.
\end{proof}

\subsubsection{On a lower bound to the number of iterations}
The number of iterations necessary to have $d(a_{k},a^{*})<\epsilon$, for any $a^{*}\in{\rm arg}\displaystyle\min_{a\in \overline{D}} f(a)\cap D$ satisfying $\displaystyle\lim_{k\to +\infty}a_{k}=a^{*}$ and any tolerance $\epsilon>0$, can be estimated under certain assumptions to the ($IPP$) algorithm whether ${\rm arg}\displaystyle\min_{a\in \overline{D}} f(a)\cap D\not=\emptyset$ of course. It is made by the following results.
\begin{proposicao}
Let $\{a_{k}\}$ be the sequence generated by the ($IPP$) algorithm. Assume that ${\rm arg}\displaystyle\min_{a\in \overline{D}} f(a)\cap D\not=\emptyset$. If  there exists $\mu_{k}\in (0,\mu]$, for any $k=0,1,\cdots$ and $\mu\in(0,1)$, for which    
\begin{equation}
\frac{\epsilon_{k}}{\beta_{k}}\leq \frac{\mu_{k}}{2}d^2(a_{k+1},a_{k}),\label{epsilon_k_gradiente3}
\end{equation}
then $$2\frac{\epsilon_{k}}{\beta_{k}}(\frac{1}{\mu}-1)\leq d^2(a_{k},a^{*}), \ \forall k=0,1,\cdots,$$
where $a^*\in {\rm arg}\displaystyle\min_{a\in \overline{D}} f(a)\cap D$ satisfies $\displaystyle\lim_{k\to+\infty}a_{k}=a^{*}$.
\label{boagaroto}
\end{proposicao}
\begin{proof}
By Proposition \ref{lema_especial},  
$$d^2(a_{k+1},a)\leq d^2(a_{k},a)-d^2(a_{k},a_{k+1})+\frac{2}{\beta_{k}}(f(a)-f(a_{k+1}))+2\frac{\epsilon_{k}}{\beta_{k}}.$$ In particular, by Replacing $a$ by $a^*$ in the inequality above and assuming (\ref{epsilon_k_gradiente3}), we have that
$$
\begin{array}{l}
d^2(a_{k},a^*)\geq d^2(a_{k+1},a^*)+d^2(a_{k},a_{k+1})+\frac{2}{\beta_{k}}(f(a_{k+1})-f(a^*))-2\frac{\epsilon_{k}}{\beta_{k}}\\
\geq d^2(a_{k+1},a^*)+2\frac{\epsilon_{k}}{\beta_{k}\mu_{k}}+\frac{2}{\beta_{k}}(f(a_{k+1})-f(a^*))-2\frac{\epsilon_{k}}{\beta_{k}}\\
\geq d^2(a_{k+1},a^*)+\frac{2}{\beta_{k}}(f(a_{k+1})-f(a^*))+2\frac{\epsilon_{k}}{\beta_{k}}(\frac{1}{\mu}-1)\\
\geq  
2\frac{\epsilon_{k}}{\beta_{k}}(\frac{1}{\mu}-1),
\end{array}
$$
and the statement follows.
\end{proof}
\begin{corolario}
Let $\{a_{k}\}$ be the sequence generated by the ($IPP$) algorithm and $\epsilon>0$. Admit that ${\rm arg}\displaystyle\min_{a\in \overline{D}} f(a)\cap D\not=\emptyset$. If  $\theta_1=1$, $\theta_2=\frac{1}{\omega}$, for any integer positive $\omega>1$, and (\ref{epsilon_k_gradiente3}) is assured, for every $k=0,1,\cdots$, then at least
\begin{equation*}
\lceil\frac{{\rm log}(2\epsilon_{0}(1-\mu))-{\rm log}(\beta_{0}\mu\epsilon)}{{\rm log}(\omega)}\rceil
\end{equation*}
iterations are necessary to have $d(a_{k},a^*)\leq\epsilon$, where $\lceil t\rceil$ represents the ceiling of $t$, for any $t\in\mathbb{R}$, and $a^{*}=\lim_{k\to+\infty}a_{k}$.\label{complexidade}
\end{corolario}
\begin{proof}
Notice that, under these hypotheses, $\beta_{k}=\beta_{0}$  and $\epsilon_{k}=(\theta_2)_{k}\cdot\epsilon_{0}$, for every $k=0, 1,\cdots$. Moreover, 
$$\frac{\theta_2}{\theta_1}=\frac{\frac{1}{\omega}}{1}=\frac{1}{\omega}<1,$$
since $\omega>1$. So, the hypotheses of Theorem \ref{convergencia_versao_inexata} is attended. This implies that $a_{k}$ converges to any point from ${\rm arg}\displaystyle\min_{a\in \overline{D}} f(a)\cap D$. Denote by $a^{*}=\displaystyle\lim_{k\to+\infty}a_{k}$. Set $d(a_{k},a^{*})\leq \epsilon$. From Proposition \ref{boagaroto}, we have that $\omega^k\geq \frac{2\epsilon_{0}(1-\mu)}{\beta_{0}\mu\epsilon}$. Therefore,
$$k\geq \frac{{\rm log}(2\epsilon_{0}(1-\mu))-{\rm log}(\beta_{0}\mu\epsilon)}{{\rm log}(\omega)}.$$
\end{proof}

\section{Practical aspects}
This section exposes a brief discussion about intrinsic features on the inexact version to the method as the nonsmooth iterative scheme to be used for computing $b_{j+1}$ and $c_{j+1}$, as defined in (\ref{passo_espectral}) and (\ref{passo_unitario}) respectively, as well as the possibilities of choice to global parameters for the algorithm.   

\subsection{Nonsmooth iterative scheme to compute $b_{j+1}$ and $c_{j+1}$}
At the first place we emphasize that the classical Riemannian relation
\begin{equation}
{\rm grad} \ h(x)\equiv [G(x)]^{-1}\nabla h(x) \label{regradrelation} 
\end{equation}
is still valid for subgradients, i.e., $[G(x)]^{-1}s$ is a Riemannian subgradient to $h$, at $x\in\mathbb{M}$, whether $s$ is any Euclidean subgradient to $h$, at $x$, for any function $h$ defined on a Riemannian manifold $\mathbb{M}\subset \mathbb{E}$. It immediately follows from the inequality in the subdifferential definition at (\ref{subdiferencial}). Based on the relation (\ref{regradrelation}) we compute by first  an approach to the negative of the derivative of $h$, at $x$, for each canonical direction in $\mathbb{E}$. It means that if $\{e_1,\cdots,e_n\}$ is the canonical base to $\mathbb{E}$ then 
$$s_{\iota}=\frac{h(x)-h(x+\delta e_{\iota})}{\delta} \ (\iota=1,\cdots, n),$$
for any $\delta>0$ small enough\footnote{For instance, $\delta$ can be chosen near to the smallest number which is representable by the computer in the arithmetic of float point.}, is an acceptable approach to the negative of the derivative of $h$, at $x$, in the direction $e_{\iota}$ ($\iota=1,\cdots,n$) whether $h$ is differentiable of course. After that, the vector $s=(s_1,\cdots,s_n)$ is a reasonable approach to $-\nabla h(x)$. So, the direction $d=[G(x)]^{-1}s$ is also a reasonable approach to $-{\rm grad} \ h(x)$, guaranteed that $h$ is differentiable. Let $\xi(x,d,\overline{t})$ be the point from the geodesic $\xi\subset \mathbb{M}$, for which $\xi(0)=x$ and $\xi'(0)=d$, associated to the parameter $\overline{t}$.  Table \ref{tabela3} resumes the iterative scheme to be employed to estimate $b_{j+1}$ and $c_{j+1}$, for any $j$ ($j=0,1,\cdots$). 

\begin{table}[htb!]
\centering
\begin{tabular}{|l|}
\hline
\textbf{input} $x\in \mathbb{M}; \delta,\sigma,\tau\in(0,1); \upsilon>1$.\\
\hline\hline
\hspace{0.3cm} \textbf{repeat}\\
\hspace{0.6cm} \textbf{for} $\iota=1$ \textbf{to} $n$ \textbf{do}\\
\hspace{0.9cm} \textbf{compute} $s_{\iota}=\frac{h(x)-h(x+\delta e_{\iota})}{\delta}$;\\
\hspace{0.6cm} \textbf{end} \\ 
\hspace{0.6cm} $d=[G(x)]^{-1}s$;\\
\hspace{0.6cm} $\overline{t}=1$;\\
\hspace{0.6cm} \textbf {if} $h(\xi(x,d,\overline{t}))\geq h(x)-\overline{t}\cdot \eta\cdot\|d\|^2_{x}$ \textbf{do}\\
\hspace{0.9cm} \textbf{while} $h(\xi(x,d,\overline{t}))\geq h(x)-\overline{t}\cdot \eta\cdot\|d\|^2_{x}$ \textbf{do}\\
\hspace{1.2cm} $\overline{t}=\frac{\overline{t}}{\upsilon}$;\\
\hspace{0.9cm} \textbf{end}\\
\hspace{0.6cm} \textbf{else}\\
\hspace{0.9cm} \textbf{while} $h(\xi(x,d,\overline{t}))< h(x)-\overline{t}\cdot \eta\cdot\|d\|^2_{x}$ \textbf{do}\\
\hspace{1.2cm} $\overline{t}=\upsilon\cdot \overline{t}$;\\
\hspace{0.9cm} \textbf{end}\\
\hspace{0.9cm} $\overline{t}=\frac{\overline{t}}{\upsilon}$;\\
\hspace{0.6cm} \textbf{end}\\
\hspace{0.6cm} $x_{{\tiny\rm aux}}=x$;\\
\hspace{0.6cm} $x=\xi(x,d,\overline{t})$;\\
\hspace{0.3cm} \textbf{until} $|h(x)-h(x_{{\tiny\rm aux}})|<\tau$\\
\textbf{end}\\
\hline
\end{tabular}
\caption{Nonsmooth iterative scheme with Armijo-like line search.}
\label{tabela3}
\end{table} 

The choice by an approximated iterative scheme results from the fact that $b_{j+1}$ and $c_{j+1}$ are intermediate steps in the general structure from the proximal point algorithm. Thus, computing $b_{j+1}$ and $c_{j+1}$ exactly can be further expensive to the global performance of the method. Besides, the main iterated $a_{k+1}$ is not computed exactly in the inexact version. 

When the objective functions in  (\ref{passo_espectral}) and (\ref{passo_unitario}) are differentiable and their Euclidean gradients are easily synthesized, by replacing the direction $d$ in Table \ref{tabela3} by $-[G(x)]^{-1}\nabla h(x)$, the resulting method is the geodesic gradient algorithm with Armijo line search in \cite{Yang}. Alternatively, when $h$ is differentiable however its Euclidean gradient is hard to be analytically  synthesized, the scheme in Table \ref{tabela3} is an approximated geodesic gradient algorithm with an Armijo-like line search.  

The stopping criteria $|h(x)-h(x_{{\tiny\rm aux}})|<\tau$ can be combined with $d(x_{{\tiny\rm aux}},x)<\tau$ when $\mathbb{M}$ is a Hadamard manifold for example. For instance, $D_1$ is a Hadamard manifold since itself is a homogeneous domain of positivity with respect to $\sigma$.  

The convexity of $h$ in despite of its differentiability implies that the direction $d$ is a Riemannian subgradient for $h$ at $x$. Actually, the iterative scheme at Table \ref{tabela3} is a Riemannian subgradient algorithm under the hypotheses of convexity and non-differentiability to $h$, with an additional improvement given by an Armijo-like line search on the direction $d$. We recommend \cite{Ferreira1} and \cite{Udriste} for further information about subgradient algorithms in Riemannian manifolds.

In a practical way, it is generally  set $\eta=0.2$ and $\upsilon=2$ in implementations involving Armijo line search as suggested by acknowledged authors. For instance, see \cite[p. 362]{Bazaraa}. In addition,  the tolerance $\tau$ can be iteratively controlled  by starting with a not as accurate  value $\tau_{0}$ and updating $\tau_{j}$, for $j=1,2,\cdots,$ with a reduction factor $\kappa\in(0,1)$. 

Notice that the scheme at Table \ref{tabela3} differs from the algorithm in \cite{Yang} by the choice of an approximated way to compute the Euclidean gradient to $h$ whether $h$ is differentiable. However, if Riemannian gradients of the objective functions in (\ref{passo_espectral}) and (\ref{passo_unitario}) are easily computed then the algorithm in \cite{Yang} is more recommended to this aim. 

\subsection{Global parameters}

The choice of global parameters for any computational method in optimization requires a great number of numerical experiments. Generally, Their values are empirically determined based on the computational behavior of the method. Nevertheless, theoretical aspects also point practical directions to be followed. For instance, notice that the conditions established by Corollary \ref{complexidade} are easy to be got. Thus, it must be taking into account to choosing $\beta_0$, $\epsilon_0$ and $\omega$, for any prefixed tolerance $\epsilon$, since proper choices can require few outer iterations to the algorithm for computing $\epsilon$-solutions.    

\section{Conclusions}
 
Domain of positivity is a little discussed conceptual structure in Mathematics although it commonly appears in several applications and it has already been consolidated in the theoretical aspect. The theory to domains of positivity initially emerges as a simple attempt to extend  properties inherent to $\mathbb{P}_n$ however the research developed by Rothaus in \cite{Rothaus} shows a robust field able to connect two important geometrical point of view, the Euclidean and the Riemannian one. Moreover, several areas as those already cited along this paper has their data encompassed by that structure. For instance, areas like Signal Processing, Computer Vision, Pattern Recognition  among others can be detached. The current paper theoretically improves and extends a proximal point technique employed to approach minimizers of convex functions defined in the closure of  $\mathbb{P}_n$ to homogeneous domains of positivity, reducible in a special sense. Two relevant domains of this nature are cited in \cite{Rothaus} and commented here. Although nonsmooth, the current technique can be applied for both differentiable and non-differentiable convex cases. Applications related to solutions of minimization problems in the two most known domains of positivity as the Riemannian mean and median justify the synthesized technique once almost every relevant method used to that goal avoid the natural Riemannian structure to domains of positivity. They alternatively use Partial Differential Equations or Log-Euclidean metrics. See \cite{Arsigny} for the second case. The implementation of the current technique as a tool for DTI computations is the main aim for future works. Several steps in this direction has already been given as a preliminary implementation of the technique for computing Riemannian averages of symmetric positive definite matrices however it is not relevant outside the application context.

\end{document}